\documentclass[12pt]{article}

\usepackage{amsfonts,amsmath,amssymb,amsthm,mathtools}
\usepackage[hyperfootnotes=false,bookmarks=false]{hyperref}
\usepackage{tikz}
\font\tenmath=msbm10 scaled 1200
\font\sevenmath=msbm7 scaled 1200
\font\Fivemath=msbm5 scaled 1200
\newfam\mathfam \textfont\mathfam=\tenmath
\scriptfont\mathfam=\sevenmath \scriptscriptfont\mathfam=\Fivemath

\newtheorem{thm}{Theorem}[section]
\newtheorem{lemma}{Lemma}[section]
\newtheorem{cor}{Corollary}[section]
\newtheorem{prop}{Proposition}[section]
\theoremstyle{definition}
\newtheorem{remark}{Remark}

\def \\ { \cr }
\def\mT{\mathcal{T}}
\def\R{\mathbb{R}}
\def \1{1 \mkern -6mu 1} 
\def\N{\mathbb{N}}
\def\P{\mathbb{P}}
\def\E{\mathbb{E}}
\def\Z{\mathbb{Z}}
\def\R{\mathbb{R}}
\def \e{{\rm e}}
\def\hZ{\hat{Z}}
\def\vareps{\varepsilon}
\def\eps{\epsilon}
\def\dt{\textup{d}}

\begin{document}

\title{On a class of random walks with reinforced memory}
\author{{Erich Baur\footnote{Bern University of Applied Sciences,
      Switzerland. Email: erich.baur@bfh.ch.}}\\}
\maketitle 
\thispagestyle{empty}
\begin{abstract}
  This paper deals with different models of random walks with a reinforced
  memory of preferential attachment type. We consider extensions of the
  Elephant Random Walk introduced by Sch\"utz and Trimper~\cite{ScTr} with
  stronger reinforcement mechanisms, where, roughly speaking, a step from
  the past is remembered proportional to some weight and then repeated with
  probability $p$. With probability $1-p$, the random walk performs a step
  independent of the past. The weight of the remembered step is increased
  by an additive factor $b\geq 0$, making it likelier to repeat the step
  again in the future.  A combination of techniques from the theory of
  urns, branching processes and $\alpha$-stable processes enables us to
  discuss the limit behavior of reinforced versions of both the Elephant
  Random Walk and its $\alpha$-stable counterpart, the so-called Shark
  Random Swim introduced by Businger~\cite{Bu1}. We establish phase
  transitions, separating subcritical from supercritical regimes.
\end{abstract}

\noindent{\bf Keywords:} Reinforced random walks, preferential attachment,
memory, stable processes, branching processes, P\'olya urns.\newline
\noindent{\bf AMS Subject Classification:} 60G50; 60G52; 60K35; 05C85.

\section{Introduction}
In the last decades, there has been a constant interest in (usually
non-Markovian) random walks with reinforcement. Arguably the most important
class is formed by \textit{edge} (or \textit{vertex}) \textit{reinforced
  random walks}. We point to the survey of Pemantle~\cite{Pe} or the more
recent works~\cite{CoTa, MaBr, SabZe} with references therein, just to
mention a few.

Loosely speaking, an edge reinforced random walk crosses an edge with a
probability proportional to a weight associated to that edge, which
increases after each visit. Edge reinforced random walks have found several
applications in statistical physics and Bayesian statistics, see,
e.g.,~\cite{DiRo, SabTar, SabZe}.

In this paper, we shall be interested in another class of random walks with
reinforcement, where at each time $n$ and with a certain probability $p$, a
step from the past is selected according to some weight (which may change
over time) and then repeated, whereas with the complementary probability
$1-p$, a new step independent of the past is performed. One of the
practical interests in such walks comes from the fact that they serve as toy
models for anomalous diffusion, describing many phenomena in physics,
chemistry and biology~\cite{MeKl, OlFeLaVa}.

A prominent example in this class is the Elephant Random Walk (ERW for
short) introduced by Sch\"utz and Trimper~\cite{ScTr} (equal weights,
symmetric $\pm 1$ steps), which has drawn a lot of attention in recent
years, see~\cite{AlArCrSiSiVi, BaBe1, Bercu, BercuLa, Be1, CoGaSc1,
  CoGaSc2, GuSt, PaEs}, though this is a non-exhaustive list. It is the
purpose of this paper to extend both the ERW and its $\alpha$-stable
version, the Shark Random Swim (SRS for short) introduced by
Businger~\cite{Bu1}, to models with a stronger (linear) reinforcement
mechanism.

Our motivation stems from the desire to describe models of preferential
attachment or ``rich get richer''-type (see~\cite{BarAl} for the origin of
such models), which should play an important role in the understanding of
evolving networks like social networks or, maybe most prominently, neural
networks. The latter are nowadays ubiquitously used to solve problems from
computer vision or machine translation and lead in combination with
reinforcement learning algorithms to stunning results, see, for example,
the recent success of AlphaGo~\cite{SiHu}.

The ERW, SRS and our extensions fit into a general framework, which we
describe first in a somewhat informal manner, referring to
Section~\ref{sec:RWreinforced} for precise definitions. We shall
propose two models of a random walk with reinforcement, which differ in
their reinforcement mechanism. 

We fix a memory parameter $p\in(0,1)$, a reinforcement parameter $b\geq 0$
measuring the strength of the reinforcement, and a sequence
$(\xi_i, i\in\N)$ of i.i.d. random variables used to model the increments
of the random walks. Moreover, we initialize time-evolving weights
$k_n(\cdot)$ by setting $k_n(i)=1$ for all $i=1,\ldots,n$ and all $n\in\N$.

The first step in both models is given by $\xi_1$.  At time $n\geq 2$, we
select one of the preceding times $I_n\in\{1,\ldots,n-1\}$ chosen at random
proportional to the weights $k_{n-1}(\cdot)$. With probability $p$, the
random walk repeats the step $\xi_{I_n}$ performed at time $I_n$, whereas
with the complementary probability $1-p$, the walk performs a new step
$\xi_n$ independent of the past.

Now in the first model which we call the {\it memory-reinforced random
  walk}, the weight of the selected time $I_n$ is updated to
$k_n(I_n)=k_{n-1}(I_n)+b$ if and only if the walk decided to repeat the
step $\xi_{I_n}$ (i.e., with probability $p$), whereas in the second model
called the {\it strongly memory-reinforced random walk}, the weight of the
selected time $I_n$ is \textit{always} updated to
$k_n(I_n)=k_{n-1}(I_n)+b$.

One may shed light on the two models from a different perspective, which is
interesting from the point of view of modeling. Namely, we may interpret
the memory-reinforced random walk as a model with certain memory lapses, in
the sense that the walk remembers and repeats a previous step only with
probability $p$. With probability $1-p$, it performs a new step (say, due
to a memory loss).

In contrast, we may view the strongly memory-reinforced walk as a model
with a perfect memory, in the sense that the random walk {\it always}
remembers a previous step, leading to a reinforcement effect in its
memory. However, then -- say, due to bad experiences with the past or due
to a unwillingness to repeat previous faults -- the walk decides to repeat
the step with probability $p$ only, whereas with probability $1-p$, it
“forgets” the past and decides to perform a new step.

We particularize these models to two choices of the sequence
$(\xi_i, i\in\N)$: For independent symmetric $\pm 1$-variables (or their
$d$-dimensional generalizations), the first model constitutes what we call
the {\it reinforced Elephant Random Walk}, whereas the second model leads
to the {\it strongly reinforced Elephant Random Walk}.

For isometric $d$-dimensional $\alpha$-stable random variables, where
$\alpha\in (0,2]$, the first and second model give rise to the {\it
  reinforced Shark Random Swim} and to the {\it strongly reinforced Shark
  Random Swim}, respectively.

The naming of these models is explained by the fact that in the case $b=0$
corresponding to a uniform choice of a previous step, we obtain the
original ERW~\cite{ScTr} and SRS~\cite{Bu1}, respectively. Although we will
always allow $b=0$ for completeness (and therefore rediscover along the way
results from~\cite{BaBe1, Bercu, Bu1, CoGaSc1}), we are here interested in the
``truly'' reinforced case $b>0$, for which both models are
non-Markovian.

The core part of this work deals with the long-time behavior of the
strongly reinforced SRS. The mere reinforced SRS would probably again require a
different approach (see Remark~\ref{rem:mere-reinforcedSRS}), which we
leave for further investigation. We shall however first discuss both the
reinforced and strongly reinforced ERW, for which we establish
representations in terms of finite-color urns. The non-Markovian nature of
the random walks is handled by keeping track of the number of times a step
is repeated. Results of Janson~\cite{Ja} on P\'olya urns then allow us to
derive the asymptotic behavior of the random walks.

For the mere reinforced ERW model covered by Theorem~\ref{thm:ERW1}, we
shall observe a phase transition at $p_{\ast}:=1/(2+b)$. In the subcritical
regime $p<p_\ast$, we prove convergence in law of the (properly normalized)
reinforced ERW towards a Gaussian limit process. In the critical regime
$p=p_\ast$, a scaled Brownian motion appears in the limit, and a non
nontrivial presumably non-Gaussian process in the supercritical regime
$p>p_\ast$.

Provided $0\leq b< 1$, the strongly reinforced ERW discussed in
Theorem~\ref{thm:ERW2} exhibits a phase transition at
$p_{\ast\ast}:=(1-b)/2$. Depending on whether $p<p_{\ast\ast}$,
$p=p_{\ast\ast}$ or $p>p_{\ast\ast}$, we obtain results similar to the
case of the mere reinforced ERW. One should note that if $b\geq 1$, then
$p_{\ast\ast}\leq 0$, implying somewhat surprisingly that in this case, for
\textit{every} choice of $p\in(0,1)$, the strongly reinforced ERW is
supercritical and behaves superdiffusively.

Changing over to the strongly reinforced SRS, we note that a finite-color
urn is inappropriate for modeling purposes, since the step variables take
infinitely many values. Although there is a growing literature on
infinite-color urns, see, e.g., the recent work~\cite{MaMa}, we follow a
different route inspired by an idea of K\"ursten~\cite{Kue}. He observed a
connection between the original ERW and clusters sizes of a Bernoulli bond
percolation on random recursive trees. His ideas were further elaborated by
Businger in~\cite{Bu1} to understand the SRS, and here, we consider
percolation on a family of \textit{preferential attachment trees} to model
the strongly reinforced SRS. One of the main difficulties compared
to~\cite{Bu1} stems from the fact that the tree processes representing the
percolation clusters are no longer branching processes. For controlling
their sizes, we are guided by ideas from Bertoin and Uribe
Bravo~\cite{BeBr}, who were interested in supercritical percolation on a
family of large preferential attachment trees. Their tree model differs
from ours, but their techniques prove useful also in our setting.

Taking inspiration from the recent work of Businger~\cite{Bu1}, we
make use of the connection to cluster sizes and prove a phase transition
for the strongly reinforced SRS at $\alpha\kappa=1$, where
$\kappa:=(b+p)/(b+1)$. More specifically, in the subcritical case
$\alpha\kappa<1$, we prove in Theorem~\ref{thm:SRSsubcritical} weak
convergence of finite-dimensional laws towards a non-L\'evy $\alpha$-stable
process. In the critical case $\alpha\kappa=1$, we establish in
Theorem~\ref{thm:SRScritical} convergence towards an $\alpha$-stable L\'evy
process. The case $\alpha\kappa>1$ treated in
Theorem~\ref{thm:SRSsupercritical} covers the supercritical regime. We
stress that for $\alpha=2$, our results show that the strongly reinforced
SRS behaves like the strongly reinforced ERW, which should
not come as a surprise.

The rest of this work is organized as follows. In
Section~\ref{sec:RWreinforced}, we introduce the general setting of
(strongly) memory-reinforced random walks, which we specify in
Section~\ref{sec:ERW} to the ERW. There we also state our main results on
the (strongly) reinforced ERW, which are proved in
Section~\ref{sec:ERWUrns} by establishing a connection to appropriate urn
models. In Section~\ref{sec:SRS} we change over to the strongly reinforced
SRS, for which we need more preparation: First, in
Section~\ref{sec:connectionPAT}, we explain the connection to (percolation
on) preferential attachment trees, which are then constructed in
continuous-time in Section~\ref{sec:continuousPAT}. By using methods from
branching processes, we gain in Section~\ref{sec:clustersizesPAT} control
over large cluster sizes, which enables us to finally discuss the
asymptotic behavior the strongly reinforced SRS in
Section~\ref{sec:asympSRS}. In Section~\ref{sec:GenPer}, we briefly discuss
some possible generalizations and perspectives. Appendix~\ref{sec:appendix}
contains the proofs of some auxiliary results on branching processes which
are used in the main part.

A final word concerning notation: For two
sequences $(r_n)$, $(s_n)$ of positive reals, we write $r_n\lesssim s_n$ if
there exists a constant $C$ (which might depend on $b$ and $p$) such that
\[r_n\leq Cs_n\quad\textup{for all }n\in\N:=\{1,2,\ldots\}.\] 
\section{Random walks with reinforced memory}
\label{sec:RWreinforced}
We fix a memory parameter $p\in (0,1)$ and
a (real-valued) reinforcement parameter $b\geq 0$. Let
$(\xi_i,i\in\N)$ be a sequence of i.i.d. random variables in
$\R^d$. This sequence will be used to model the steps of the random
walk. (Later on, we will consider the variables $\xi_i$ under two particular
laws.)

We further let $(\eps_i,i\geq 2)$ be a sequence of i.i.d. Bernoulli
random variables with success probability $p$. Call a time $i\geq 2$ a {\it
  memory time} if $\eps_i=1$, and a {\it fresh time} if $\eps_i=0$. A
memory time will correspond to a time where the random walk repeats one of
its preceding steps, whereas a fresh time will represent a time where the
random walk performs independently of the past a new step.

We shall consider two models of a random walk with reinforced memory.  Both
models depend on time-evolving weights $k_n(\cdot)$, which we initialize by
setting
\[k_n(i)=1\quad\textup{for } i=1,\ldots,n\quad\textup{and all }n\in\N.\]

We specify the random walks by defining their increments
$\zeta_i$, $i\geq 1$. First, we set $\zeta_1=\xi_1$, and then for $n\geq
2$, we select a previous time
$I_n\in\{1,\ldots,n-1\}$ according to
\[\P(I_n=i)=\frac{k_{n-1}(i)}{\sum_{j=1}^{n-1}k_{n-1}(j)},\quad
  i=1,\ldots,n-1.\]
Now if $n$ is a memory time (i.e., if $\eps_n=1$), we let
$\zeta_n=\zeta_{I_n}$, whereas if $n$ is a fresh time (i.e., if
$\eps_n=0$), we let $\zeta_n=\xi_n$.

It remains to update the weights $k_n(\cdot)$, and here, the difference
between the two models comes into play: In the first model, we update the
weight of the selected time $I_n=i$ if and only if $n$ is a memory time
(i.e., if $\eps_n=1$), by setting
\begin{equation}
  \label{eq:weightupdate}
  k_n(i)=k_{n-1}(i)+b.\end{equation}
In the second model, we \textit{always} update the weight of the selected time $I_n=i$
according to~\eqref{eq:weightupdate}, no matter whether $n$ is a memory
time or not. The other weights remain unchanged, in both models.

Letting 
\[S_n:=\zeta_1+\ldots+\zeta_n,\quad n\in\N,\] we call the process
$(S_n, n\in\N)$ either the \textit{memory-reinforced random walk} or the
\textit{strongly memory-reinforced random walk}, depending on which update
rule is applied (first or second model).

To summarize in words, if $n$ is a fresh time, the random walk performs a step
independently of the past, whereas if $n$ is a memory time, the walk
repeats the step performed at time $I_n$. In the memory-reinforced case,
the weight $k_n(I_n)$ of the selected time is increased by the amount $b$
if and only if $n$ is a memory time, whereas in the strongly
memory-reinforced case, $k_n(I_n)$ is always increased by the amount $b$.

\begin{remark}  
  If $b=0$, then the weights remain equal to one, and both models agree.
  More precisely, in case of a memory time, the increment $\zeta_n$ is
  chosen uniformly at random among the previous increments -- our
  memory-reinforced random walk thus corresponds for $b=0$ to what is
  called {\it step reinforced random walk} in~\cite{Be1}. More generally, the
  parameter $b$ captures the strength of the reinforcement: The larger $b$
  is, the heavier the weight of the chosen time becomes, and the likelier
  it is to remember this time again in the future.
\end{remark}

\section{The (strongly) reinforced Elephant Random Walk}
\label{sec:ERW}
By choosing $(\xi_i,i\in\N)$ to be an i.i.d. sequence with law
\[\P\left(\xi_1=1\right)=\P\left(\xi_1=-1\right)=\frac{1}{2},\]
the two models of reinforcement described in the last section give rise to
what we call the \textit{reinforced Elephant Random Walk} and the
\textit{strongly reinforced Elephant Random Walk}, respectively (or, for
short, the (strongly) reinforced ERW). 

To make a clear distinction to the strongly reinforced ERW, we shall
sometimes refer to the first model as the ``mere reinforced ERW''. As we
further explain in Section~\ref{sec:GenPer}, the fact that we consider here
only one-dimensional symmetric $\pm 1$-steps $\xi_i$ is just to keep the
presentation simple.

We stress that in the case of our interest $b>0$, both ERW with
reinforcement are non-Markovian even in dimension $d=1$. This is in
contrast to the original one-dimensional ERW corresponding to $b=0$:
Indeed, there, if the elephant is at position $k\in\Z$ at time $n$, then it
performed $(n+k)/2$ steps to the right and $(n-k)/2$ steps to the left up,
and more information from the past is irrelevant for predicting the
$(n+1)$th step. (In dimensions greater or equal to two, the isotropic ERW
model is non-Markovian for any $b\geq 0$.) For results on the original
multi-dimensional ERW based on a martingale approach, see Bercu and Laulin~\cite{BercuLa}.

We state now our main results describing the limiting behavior of the
reinforced elephants in the Skorokhod space $D([0,\infty))$ of
right-continuous functions with left-hand limits. We write
$\overset{\text{d}}\rightarrow$ or $\overset{\text{a.s.}}\rightarrow$ for
convergence in law or almost sure convergence, respectively. We
start with the mere reinforced model. 
\begin{thm}
\label{thm:ERW1}
Let $p\in (0,1)$, $b\geq 0$, and let $(S_n,n\in\N)$ be the reinforced
ERW with parameters $b$ and $p$. Moreover, set $p_\ast:=1/(2+b)$, and let
$\kappa:=(b+1)p/(bp+1)$. Then the following convergences hold for $n\rightarrow\infty$: 
  \begin{itemize}
  \item[a)]{\textsc{Subcritical case:}} If $p<p_\ast$,
\[
\left(\frac{S_{\lfloor tn\rfloor}}{\sqrt{n}}, t\geq 0\right)\overset{\text{d}}\longrightarrow
  (W_t,t\geq 0),
\]
where $(W_t,t\geq 0)$ is a continuous $\R$-valued mean-zero Gaussian
process started from $W_0=0$, with covariances
\[\E\left[W_sW_t\right]=\frac{bp+1}{(1-(2+b)p)(b+1)}s\left(\frac{t}{s}\right)^{\kappa}+ \frac{(pb^3+(3p-p^2)b^2+b)}{(bp+1)^2(b+1)}s
 ,\quad 0<s\leq t.\]
\item[b)]{\textsc{Critical case:}} If $p=p_\ast$,
\[
\left(\frac{S_{\lfloor n^t\rfloor}}{\sqrt{n^t\ln n}}, t\geq 0\right)\overset{\text{d}}\longrightarrow
  \sqrt{\frac{p}{1-p}}(B_t,t\geq 0),
\]
where $(B_t,t\geq 0)$ is a one-dimensional Brownian motion.

  \item[c)]{\textsc{Supercritical case:}} If $p>p_\ast$,
\[
  \left(\frac{S_{\lfloor tn\rfloor}}{n^{\kappa}}, t\geq
    0\right)\overset{\text{a.s.}}\longrightarrow\left(t^{\kappa}
    Y,t\geq 0\right)
\]
for some nontrivial random variable $Y=Y(b,p)$.
  \end{itemize}  
\end{thm}
\begin{remark}
\label{rem:thmERW1}
As already mentioned, the case $b=0$ corresponds to the original Elephant
Random Walk, and we recover results from~\cite{BaBe1,CoGaSc1}. (However, we
stress that the memory is differently parameterized in the cited papers,
namely by $q=(p+1)/2$.) In particular, the expression $g(p,0)$ in the
subcritical case simplifies to \[g(p,0)=\frac{1}{1-2p}=\frac{1}{3-4q}.\]

In the supercritical case $c)$, we have $\kappa>1/2$. If $b=0$, the
limiting random variable $Y$ under $c)$ is known to be non-Gaussian
(see~\cite{Bercu}), and we strongly suspect that for $b>0$, this is the
case, too. In this regard, we point to Remark 3.20 of~\cite{Ja} and to
Theorem 3.26 therein, where some information on the moments of $Y$ is
given. It seems, however, unclear how to calculate them explicitly. 
\end{remark}

For the strongly reinforced model, we obtain:
\begin{thm} 
  \label{thm:ERW2}
Let $p\in (0,1)$, $b\geq 0$, and let $(S_n,n\in\N)$ be the strongly reinforced
ERW with parameters $b$ and $p$. Moreover, set $p_{\ast\ast}:=(1-b)/2$, and let
$\kappa:=(b+p)/(b+1)$. Then the following convergences hold for $n\rightarrow\infty$: 
  \begin{itemize}
  \item[a)]{\textsc{Subcritical case:}} If $p<p_{\ast\ast}$,
\[
\left(\frac{S_{\lfloor tn\rfloor}}{\sqrt{n}}, t\geq 0\right)\overset{\text{d}}\longrightarrow
  (W_t,t\geq 0),
\]
where $(W_t,t\geq 0)$ is a continuous $\R$-valued mean-zero Gaussian
process started from $W_0=0$, with covariances
\[\E\left[W_sW_t\right]=\frac{(1-b^2)p}{(1-b-2p)(b+p)}s\left(\frac{t}{s}\right)^{\kappa}+\frac{(1+p)b}{b+p}s
 ,\quad 0<s\leq t.\]
  \item[b)]{\textsc{Critical case:}} If $p=p_{\ast\ast}$,
\[
\left(\frac{S_{\lfloor n^t\rfloor}}{\sqrt{n^t\ln n}}, t\geq 0\right)\overset{\text{d}}\longrightarrow\sqrt{\frac{2p^2}{1-p}}(B_t,t\geq 0),
\]
where $(B_t,t\geq 0)$ is a one-dimensional Brownian motion.

  \item[c)]{\textsc{Supercritical case:}} If $p>p_{\ast\ast}$,
\[
  \left(\frac{S_{\lfloor tn\rfloor}}{n^{\kappa}}, t\geq
    0\right)\overset{\text{a.s.}}\longrightarrow \left(t^{\kappa} Y,t\geq 0\right)
\]
for some nontrivial random variable $Y=Y(b,p)$.
  \end{itemize}  
\end{thm}
\begin{remark} In the supercritical case $c)$, $\kappa>1/2$, and we suspect
  the limiting random variable $Y$ again to be non-Gaussian, see
  Remark~\ref{rem:thmERW1}. Note that when $b\geq 1$, we have
  $p_{\ast\ast}\leq 0$, so that for \textit{each} choice of $p\in(0,1)$,
  case $c)$ applies. In other words, if $b\geq 1$, the strongly reinforced
  ERW behaves always superdiffusively. Informally, if we set $b=\infty$,
  then $\kappa=1$, and the elephant goes deterministically in the direction
  of its first step. Note also that in the case $b=0$, the expressions for
  the covariances under $a)$ and $b)$ agree indeed with those given in
  Theorem~\ref{thm:ERW1}.
\end{remark}
In the following section, we depict the connection of the reinforced
elephants to urn models, which will enable us to prove the above theorems.

\subsection{ Three-color urns with random replacement}
\label{sec:ERWUrns}
For what follows, we always fix parameters $p\in(0,1)$ and $b\geq 0$
without mentioning this every time.  In the description of the following
urn models we use the terminology (and often notation) of Janson~\cite{Ja},
to which we refer for more details on urns.
\begin{remark}
\label{rem:urns}
For ease of understanding, we will phrase our descriptions as if the
reinforcement parameter $b$ were a positive integer. This allows us to
interpret $b$ as a number of balls, whereas otherwise, we would have to
talk about urns containing a certain mass of each color, rather than entire
balls. In any case, all the results on urns we are going to use in the
proofs of our Theorems~\ref{thm:ERW1} and~\ref{thm:ERW2} are also correct
for arbitrary positive reals $b$, see~\cite[Remark 4.2]{Ja}.
\end{remark}

\subsubsection{A model for the reinforced ERW}
We consider an urn $X_n=(B_n, G_n, R_n)'$, $n\in\N$, with balls of types
$1$ (``black''), $2$ (``green'') and $3$ (``red''), so that $X_n$ captures
the number of black, green and red balls after $n-1$ draws,
$n\geq 1$. (We write $'$ for the transpose of a vector -- our vectors are always column
vectors.)

The urn evolves as a Markov process as follows: At time $n\geq 1$, a ball
is drawn from the urn uniformly at random and put back to the urn,
together with a random number of (new) black, green and red balls. More
specifically, if a black ball is drawn in the $n$th step, then with
probability $p$, $b+1$ black balls are added to the urn, whereas with the
complementary probability $1-p$,
a green or a red ball is added with probability $1/2$ each.

The dynamics are the same if a green ball is drawn. Last, if in the $n$th
step a red ball is drawn, then, with probability $p$, $b+1$ red balls are
added to the urn, whereas with the complementary probability $1-p$, a green
or a red ball is added with probability $1/2$ each.

The described dynamics lead to the mean replacement matrix
\begin{equation}
\label{eq:mrepm1}
A=\begin{pmatrix} (b+1)p & (b+1)p & 0 \\
      (1-p)/2 & (1-p)/2 & (1-p)/2 \\
  (1-p)/2 & (1-p)/2 & (b+1)p+(1-p)/2
   \end{pmatrix}.
 \end{equation}
 Here, the entry $(i,j)$ of the matrix $A$ captures the mean number of
 balls of type $i$, which are added to the urn if in the $n$th step a ball
 of type $j$ is drawn. (Since every drawn ball is returned to
 the urn as well, the name ``mean replacement matrix'' might 
 look a bit irritating, but it is standard in this context.)

 Let us now make the link to the reinforced ERW. Roughly speaking, an increase
 by $b+1$ of the number of black balls represents a step to the right of
 the reinforced ERW due to a memory time, green balls represent steps to
 the right due to a fresh time, and an increase by $b+1$ or by $1$ of the
 number of red balls represents a step to the left due to a memory time or due
 to a fresh time, respectively.

 More precisely, when we start the urn at time one with the random initial
 configuration consisting of one green or red ball with probability $1/2$
 each, the number of steps to the right of the reinforced ERW until time
 $n$ is distributed as
\[\frac{B_n}{b+1}+G_n.\]
In other words, we have for the position $S_n$ of the reinforced ERW
at time $n$ that
\begin{equation}
\label{eq:ERWreinforced}
S_n=_d 2\left(\frac{B_n}{b+1}+G_n\right) -n.
\end{equation}
Of course, the last display may be strengthened to an equality in law of
processes (in $n$).  With this correspondence at hand, we are in position to
prove Theorem~\ref{thm:ERW1}.
\begin{proof}[Proof of Theorem~\ref{thm:ERW1}]
The proofs are essentially consequences of results in~\cite{Ja}, but the
calculations are a bit involved. First, we find that the eigenvalues of
the mean replacement matrix~\eqref{eq:mrepm1} are given by
$\lambda_1=bp+1$, $\lambda_2=(b+1)p$, and $\lambda_3=0$. Corresponding
right eigenvectors with $L^1$-norm equal to one are
\[v_1=\frac{1}{2(bp+1)}\left((b+1)p,1-p,bp+1\right)',\quad
  v_2=\frac{1}{2}(-1,0,1)',\quad v_3=\frac{1}{2}(-1,1,0)'.\]
A corresponding dual basis of left eigenvectors $u_1,u_2,u_3$
(i.e., $u'_i\cdot v_j=\delta_{ij}$) is given by
\[u_1=(1,1,1)',\quad u_2=(-1,-1,1)',\quad
  u_3=\frac{1}{bp+1}(p-1,(2b+1)p+1,p-1)'.\] Solving the equation
$\frac{\lambda_2}{\lambda_1}=\frac{1}{2}$, we find that a phase transition
occurs at $p_\ast=1/(2+b)$, see~\cite[p. 183]{Ja}.  \newline $a)$ As for
the subcritical case $p<p_\ast$, it is readily checked that we are in the
setting of~\cite[Theorem 3.31(i)]{Ja}. We deduce that
\[\left(n^{-1/2}(X_{\lfloor tn\rfloor}-tn\lambda_1v_1), t\geq 0\right)\]
converges in distribution towards a continuous $\R^3$-valued mean-zero
Gaussian process $V=(V_t,t\geq 0)$ with $V_0=0$. In order to analyze the
covariance structure of $V$, we first note that the mean number of balls
which is added to the urn from one step to the next equals $m=bp+1$. Remark
$5.7$ in~\cite{Ja} then implies that
\begin{equation}
\label{eq:thm1VsVt-1}
\E\left[V_sV_t'\right]=(bp+1)s\Sigma\e^{\frac{\ln(t/s)}{bp+1}A'},\quad
0<s\leq t,
\end{equation}
where $\Sigma$ is the $3\times 3$-matrix given by
\[\Sigma:=\int_0^\infty (P_{\lambda_2}+P_{\lambda_3})\e^{sA}B\e^{sA'}(P_{\lambda_2}+P_{\lambda_3})'\e^{-(bp+1)s}\dt s.\]
Here, $P_{\lambda_2}=v_2u'_2$ and $P_{\lambda_3}= v_3u'_3$ are the projections onto the sum of the
generalized eigenspaces corresponding to $\lambda_2$ and $\lambda_3$, and
\begin{equation}
\label{eq:Bmatrix}
  B:=v_{11}\E\left[\theta_1\theta'_1\right]+v_{12}\E\left[\theta_2\theta'_2\right]+v_{13}\E\left[\theta_3\theta'_3\right]=\begin{pmatrix}\frac{(b+1)^2p}{2}&0&0\\
  0&\frac{1-p}{2}&0\\
  0&0&\frac{(b+1)^2p}{2}+\frac{1-p}{2}\end{pmatrix},\end{equation}
where $v_1=(v_{11},v_{12},v_{13})'$ and
$\theta_{j}=(\theta_{1j},\theta_{2j},\theta_{3j})'$ is the (random)
vector specifying how many balls of type 1 (``black''), 2 (``green'') and
3 (``red'') are added to the urn if a ball of type $j\in\{1,2,3\}$ is drawn.

Using that
$P_{\lambda_2}\e^{sA}=\e^{\lambda_2s}P_{\lambda_2}$ and similarly for
$P_{\lambda_3}$,  we integrate and obtain
\[\Sigma=\frac{1}{1-(b+2)p}P_{\lambda_2}BP'_{\lambda_2}+\frac{1}{1-p}\left(P_{\lambda_2}BP'_{\lambda_3}+P_{\lambda_3}BP'_{\lambda_2}\right)+\frac{1}{bp+1}P_{\lambda_3}BP'_{\lambda_3}.\]
Going back to~\eqref{eq:thm1VsVt-1}, it follows that for $0<s\leq t$
\begin{equation}
\label{eq:thm1VsVt-2}
\E\left[V_sV_t'\right]=(bp+1)s\left(\left(\left(\frac{t}{s}\right)^{\frac{(b+1)p}{bp+1}}\frac{P_{\lambda_2}BP'_{\lambda_2}}{1-(b+2)p}+\frac{P_{\lambda_3}BP'_{\lambda_2}}{1-p}\right)+\frac{P_{\lambda_3}BP'_{\lambda_3}}{bp+1}+\frac{P_{\lambda_2}BP'_{\lambda_3}}{1-p}\right).
\end{equation}
By~\eqref{eq:ERWreinforced}, we have 
\begin{align*}
  S_{\lfloor tn\rfloor}&=2\left(\frac{B_{\lfloor
                         tn\rfloor}}{b+1}+G_{\lfloor tn\rfloor}\right)-\lfloor
                         tn\rfloor=2\left(\frac{B_{\lfloor tn\rfloor}+p(b+1)\lfloor
                         tn\rfloor/2}{b+1}+G_{\lfloor tn\rfloor}-\frac{(1-p)\lfloor
                         tn\rfloor}{2}\right)\\
                       &=2\left(\frac{B_{\lfloor
                         tn\rfloor}-\lfloor
                         tn\rfloor\lambda_1v_{11}}{b+1}+G_{\lfloor
                         tn\rfloor}-\lfloor
                         tn\rfloor\lambda_1v_{12}\right).
\end{align*}
By the continuous mapping theorem, we deduce that
$(n^ {-1/2}S_{\lfloor tn\rfloor}, t\geq 0)$ converges in law in
$D([0,\infty))$ to a process $W=(W_t,t\geq 0)$ given by
\[
  W_t=2\left(\frac{1}{b+1}V_t^{(1)}+V_t^{(2)}\right),\] where $V_t^{(i)}$ denotes the
$i$th component of $V_t$. In particular,
\[
\E\left[W_sW_t\right]=\frac{4}{(b+1)^2}\E\left[V^{(1)}_sV^{(1)}_t\right]+
\frac{4}{b+1}\E\left[V_s^{(1)}V_t^{(2)}\right]+\frac{4}{b+1}\E\left[V_s^{(2)}V_t^{(1)}\right]
+4\E\left[V_s^{(2)}V_t^{(2)}\right].\]
Upon evaluating the matrix products $P_{\lambda_i}BP'_{\lambda_j}$ for
$i,j\in\{2,3\}$, the claim under $a)$ now follows from a small calculation using~\eqref{eq:thm1VsVt-2}.
\newline
$b)$ In the critical case $p=p_\ast$, applying~\cite[Theorem 3.31(ii)]{Ja}, we deduce that
  \[((n^t\ln n)^{-1/2}(X_{\lfloor n^t\rfloor}-n^t\lambda_1v_1), t\geq
  0)\]
  converges in law as $n\rightarrow\infty$ towards a continuous $\R^3$-valued mean-zero Gaussian process
  $\tilde{V}=(\tilde{V}_t,t\geq 0)$ with $\tilde{V}_0=0$ and covariance matrix
  \[\E\left[\tilde{V}_s{\tilde{V}_t}'\right]=\left(P_{\lambda_2}BP'_{\lambda_2}\right)s=\frac{1}{4p}\begin{pmatrix}1-p&0&-(1-p)\\0&0&0\\-(1-p)&0&1-p\end{pmatrix}s,\] where we have used the critical relation
  $b=(1/p)-2$. The limiting process $\tilde{W}=(\tilde{W}_t,t\geq 0)$ of
  $((n^t\ln n)^{-1/2}S_{\lfloor n^t\rfloor},t\geq 0)$ is related to
  $\tilde{V}$ by
  $\tilde{W}_t=2(\frac{1}{b+1}\tilde{V}_t^{(1)}+\tilde{V}_t^{(2)})$. From
  this and the last display, the claim readily follows.  \newline $c)$ In
  the supercritical case $p>p_\ast$, using~\cite[Theorem 3.24]{Ja}, we see
  that
 \[\left(n^{-\kappa}(X_{\lfloor tn\rfloor}-tn\lambda_1v_1), t\geq 0\right)\]
 converges almost surely to $(t^{\kappa}\hat{W}, t\geq 0)$, where
 $\hat{W}=(\hat{W}_1,\hat{W}_2,\hat{W}_3)'$ is a (nonzero) random vector in
 the eigenspace of $A$ associated to $\lambda_2$. Claim $c)$ now follows, with
  $Y=2(\frac{1}{b+1}\hat{W}_1+\hat{W}_2)$.
\end{proof}

\subsubsection{A model for the strongly reinforced ERW}
Similarly to the last section, the strongly reinforced ERW may be modeled
in terms of a three-color urn $X_n=(B_n, G_n, R_n)'$, $n\in\N$, with random
replacement.

However, in this model, the interpretation of balls of different colors
will not be the same as in the urn model from the last section. Here, the
number of black balls increases by $b$ if a step to the right was
remembered, green balls represent steps to the right, and an increase of
the number of red balls means that a step to the left was remembered and/or
performed.

More precisely, if in the $n$th step a black is drawn, always $b$ black
balls are added to the urn. In addition, with probability $p$,
a green ball is added to the urn, representing the
event that the elephant repeats the remembered step to the right. With the
complementary probability $1-p$,
a green or a red ball is added with probability $1/2$ each.

The same dynamics apply to the case when a green ball is drawn. Finally, if
in the $n$th step a red ball is drawn, then always $b$ red balls are
added; moreover, with probability $p$, another red ball is added, whereas
with the complementary probability $1-p$, another red ball or a green ball
is added with probability $1/2$ each.

In this case, we arrive at the mean replacement matrix
\begin{equation}
\label{eq:mrepm2}
A=\begin{pmatrix} b & b & 0 \\
      (1+p)/2 & (1+p)/2 & (1-p)/2 \\
  (1-p)/2 & (1-p)/2 & b+(1+p)/2
   \end{pmatrix}.
\end{equation}

Assuming again that we start the process at time $1$ with one green or one
red ball with equal probability, the number of steps to the right of the
strongly reinforced ERW until
time $n$ is now modeled by the number of green balls $G_n$. Therefore, the
position $S_n$ of the strongly reinforced ERW at time $n$ satisfies 
\begin{equation}
\label{eq:ERWstronglyreinforced}
S_n=_d 2G_n-n.
\end{equation}

We proceed now to the proof of Theorem~\ref{thm:ERW2}. 
\begin{proof}[Proof of Theorem~\ref{thm:ERW2}]
  We always refer to the urn with mean replacement matrix $A$ specified
  in~\eqref{eq:mrepm2}. We may assume that $b>0$, since the case $b=0$ is
  already covered by Theorem~\ref{thm:ERW1}. The eigenvalues of $A$ are
  given by $\lambda_1=b+1$, $\lambda_2=b+p$, and $\lambda_3=0$. Right
  eigenvectors of $L^1$-norm equals one corresponding to $\lambda_1$,
  $\lambda_2$ and $\lambda_3$, respectively, are
  \[v_1=\frac{1}{2(b+1)}\left(b,1,b+1\right)',\quad
    v_2=\frac{1}{2(b+p)}\left(b,p,-(b+p)\right)',\quad
    v_3=\frac{1}{2}\left(1,-1,0\right)',\] A dual basis of the
  corresponding left eigenvectors is given by
  \[u_1=(1,1,1)',\quad u_2=(1,1,-1)'\]
  and
  \[u_3=\frac{1}{(b+1)(b+p)}\big((1+p)b+2p,-(2b+1+p)b,(1-p)b\big)'.\]
  Solving $\frac{\lambda_2}{\lambda_1}=\frac{1}{2}$, we find
  $p_{\ast\ast}=(1-b)/2$.
  \newline $a)$ In the subcritical case $p<p_{\ast\ast}$ (which is only
  possible if $b<1$), we may again apply~\cite[Theorem 3.31(i)]{Ja} to
  obtain convergence of
\[\left(n^{-1/2}(X_{\lfloor tn\rfloor}-tn\lambda_1v_1), t\geq 0\right)\]
towards a continuous $\R^3$-valued mean-zero
Gaussian process $V=(V_t,t\geq 0)$ with $V_0=0$. In each step, $b+1$
balls are added to the urn. Similarly to the proof of the subcritical case
in Theorem~\ref{thm:ERW1}, we compute for $0<s\leq t$
\begin{equation}
  \label{eq:VsVt-3}
  \E\left[V_sV_t'\right]=
(b+1)s\left(\left(\left(\frac{t}{s}\right)^{\frac{b+p}{b+1}}\frac{P_{\lambda_2}BP'_{\lambda_2}}{1-b-2p}+\frac{P_{\lambda_3}BP'_{\lambda_2}}{1-p}\right)+\frac{P_{\lambda_3}BP'_{\lambda_3}}{b+1}+\frac{P_{\lambda_2}BP'_{\lambda_3}}{1-p}\right).
\end{equation}
where $P_{\lambda_2}=v_2u'_2$, $P_{\lambda_3}=v_3u'_3$, and $B$ is defined analogously to~\eqref{eq:Bmatrix}, yielding
\[B=\frac{1}{4}\begin{pmatrix}2b^2 & b(1+p)& b(1-p)\\b(1+p) & 2 &
    b(1-p)\\b(1-p)& b(1-p) & 2(b^2+(1+p)b+1)
\end{pmatrix}.\]
By~\eqref{eq:ERWstronglyreinforced}, we have 
\begin{equation}
  \label{eq:ERWstronglyreinforced2}
  S_{\lfloor tn\rfloor}=2G_{\lfloor tn\rfloor}-\lfloor
  tn\rfloor=2\left(G_{\lfloor tn\rfloor}-\lfloor
    tn\rfloor\lambda_1v_{12}\right).\end{equation} Thus,
$(n^ {-1/2}S_{\lfloor tn\rfloor}, t\geq 0)$ converges in law in
$D([0,\infty))$ to a process $W=(W_t,t\geq 0)$ given by $W_t=2V_t^{(2)}$.
The claim then follows from~\eqref{eq:VsVt-3}.
\newline
$b)$ In the critical case $p=p_{\ast\ast}$, it follows from~\cite[Theorem 3.31(ii)]{Ja} that
  \[((n^t\ln n)^{-1/2}(X_{\lfloor n^t\rfloor}-n^t\lambda_1v_1), t\geq
  0)\]
  converges in law as $n\rightarrow\infty$ towards a continuous $\R^3$-valued mean-zero Gaussian process
  $\tilde{V}=(\tilde{V}_t,t\geq 0)$ with $\tilde{V}_0=0$ and covariance matrix 
  \[
    \E\left[\tilde{V}_s{\tilde{V}_t}'\right]=\left(P_{\lambda_2}BP'_{\lambda_2}\right)s=
    \begin{pmatrix}\frac{(1-2p)^2}{2(1-p)} &
      \frac{p(1-2p)}{2(1-p)}&\frac{2p-1}{2}\\
      \frac{p(1-2p)}{2(1-p)}&\frac{p^2}{2(1-p)}&-\frac{p}{2}\\
      \frac{2p-1}{2} & -\frac{p}{2} & \frac{1-p}{2}\end{pmatrix}s,\quad
    0<s\leq t,
  \]
where we have used that in the critical case $b=1-2p$. For the limiting process $\tilde{W}=(\tilde{W}_t,t\geq 0)$ of
$((n^t\ln n)^{-1/2}S_{\lfloor n^t\rfloor},t\geq 0)$, it remains to
observe that
$\tilde{W}_t=2\tilde{V}_t^{(2)}$, as under $a)$. 
\newline
$c)$ In the supercritical case $p>p_{\ast\ast}$, similarly to part $c)$ of
Theorem~\ref{thm:ERW1}, we may apply~\cite[Theorem 3.24]{Ja} to deduce
that
\begin{equation}\label{eq:convX2}
  \left(n^{-\kappa}(X_{\lfloor tn\rfloor}-tn\lambda_1v_1), t\geq 0\right)\end{equation}
converges almost surely to $(t^{\kappa}\hat{W}, t\geq 0)$ as $n$ tends to infinity, where
$\hat{W}=(\hat{W}_1,\hat{W}_2,\hat{W}_3)'$ is a (nonzero) random vector
lying in the eigenspace of $A$ associated to $\lambda_2$. Using~\eqref{eq:ERWstronglyreinforced2}
and~\eqref{eq:convX2}, the claim follows with $Y=2\hat{W}_2$.
\end{proof}

\section{The strongly reinforced Shark Random Swim}
\label{sec:SRS}
Instead of independent $\pm 1$-steps, we shall consider in this section an
i.i.d. sequence $(\xi_i, i\in\N)$ of $\R^d$-valued isotropic stable random
variables specified by
\begin{equation}
\label{eq:charStable}
\E\left[\e^{i\langle\theta,\xi_1\rangle}\right]=\e^{-\|\theta\|^{\alpha}},\quad\theta\in\R^d,\end{equation}
where the stability parameter $\alpha$ takes values in $(0,2]$.

If $d=1$, this simply means that the $\xi_i$'s are symmetric
$\alpha$-stable random variables with scale parameter one. Our arguments
are however not limited to the one-dimensional case.

Under the above sequence of $\alpha$-stable random variables, the
corresponding (strongly) memory-reinforced random walk $(S_n,n\in\N)$ gives rise to what we
call the \textit{(strongly) reinforced Shark Random Swim}, the 
(strongly) reinforced SRS for short.

In order to model how many times a certain step is repeated, we will
establish a connection to percolation on a family of preferential
attachment trees. As we explain in Remark~\ref{rem:mere-reinforcedSRS},
this technique is primarily well-suited for modeling the \textit{strongly} reinforced
SRS, to which we restrict ourselves from now on.

The asymptotic behavior of the strongly reinforced SRS will depend on how
the stability parameter $\alpha$ relates to the parameter
 \begin{equation}
 \label{eq:defKappa}
 \kappa=\kappa(b,p)= \frac{b+p}{b+1},\end{equation}
which we fix from now on once for all in this way. We point out that the same parameter appears also in
Theorem~\ref{thm:ERW2}, which should be compared with our results for the
strongly reinforced SRS in the case $\alpha=2$.

Unless stated otherwise, we shall again assume that $p\in(0,1)$ and
$b\geq 0$. In the case $b=0$, we will come across results of
Businger~\cite{Bu1} (see also her extended version
\url{https://arxiv.org/abs/1710.05671} for finite-dimensional convergence).
\subsection{Connection to preferential attachment trees}
\label{sec:connectionPAT}
We will construct on the positive integers $\N$ an increasing tree, which
follows a preferential attachment mechanism. In order to clearly illustrate the
connection to the strongly reinforced SRS, we first give a discrete-time
construction of our tree, although we shall work later on primarily with a
continuous-time construction (see Section~\ref{sec:continuousPAT}).

In order to describe the building dynamics,
we use the same weights $k_n(\cdot)$ as for the memory-reinforced random walk
models, starting from
\[k_n(i)=1\quad\textup{for } i=1,\ldots,n\quad\textup{and all }n\in\N.\]
We denote by $T_1$ the tree with a single node labeled $1$, to which we
attach a half-edge, see Figure~\ref{fig:ClustersSpins}, so that the degree
of the root node $1$ is equal to one at the beginning.

Then, for $n\geq 2$, given $\mT_{n-1}$ has been built, we attach node $n$ to
a randomly chosen node $I_n\in\{1,\ldots,n-1\}$ from the tree $\mT_{n-1}$
according to
\begin{equation}
\label{eq:SelectRuleShark}
 \P(I_n=i)=\frac{k_{n-1}(i)}{\sum_{j=1}^{n-1}k_{n-1}(j)},\quad i=1,\ldots,n-1.\end{equation}
Finally, we update the weight of the parent node $I_n=i$ of $n$ by setting
\[k_n(i)=k_{n-1}(i)+b.\] 
We note that in the case $b=0$, the above construction produces a
tree $\mT_n$ uniformly distributed among all increasing trees on the integers
$1,\ldots,n$, a so-called \textit{random recursive tree}.

The weights are intimately related to the degree of a vertex, an
observation which will be crucial in what follows. Namely, denote by $d_n(i)$ the
degree of vertex $i\in\{1,\ldots,n\}$ in $\mT_n$, i.e., the number of edges
with endpoint $i$. Then it holds that
\begin{equation}
  \label{eq:kndn}k_n(i)= b(d_n(i)-1) +1.
\end{equation}
Note that this relation is also true for the root node $1$, thanks to the
half-edge attached to it.

We now superpose a Bernoulli bond percolation with parameter $0<p<1$ on
$\mT_n$. However, following an idea from~\cite{BeBr}, rather than deleting
edges, we shall cut each edge of $\mT_n$ at its midpoint with probability
$1-p$, independently of the others edges. Writing $\mT^{(p)}_n$ for the resulting
combinatorial structure at time $n$, $\mT^{(p)}_n$ is a forest consisting of
trees with edges and half-edges. We call these trees \textit{percolation
  clusters} of $\mT^{(p)}_n$.

We write $c_{1,n},c_{2,n},\ldots$ for the sequence of percolation
clusters increasingly ordered according to the label of their root node,
with $c_{i,n}:=\emptyset$ if $\mT^{(p)}_n$ contains less than $i$
clusters. See the right hand side of Figure~\ref{fig:ClustersSpins}. In
particular, $c_{1,n}$ is the (root) cluster rooted at node $1$.  Of course,
we should rather write $c_{i,n}^{(p)}$, but we drop $p$ from the
notation. Moreover, we write $|c_{i,n}|$ for the size of the $i$th cluster,
i.e., the number of its nodes.

To make the connection to the strongly reinforced SRS, we assign
additionally ``spins'' to the nodes of $\mT_n^{(p)}$, following an idea of
K\"ursten~\cite{Kue}. More precisely, as it is indicated on the right hand 
side of Figure~\ref{fig:ClustersSpins}, we equip all the nodes of the
$i$th cluster $c_{i,n}$ with spin $\xi_i$, with $(\xi_i,i\in\N)$ a
sequence of i.i.d. stable random variables with
characteristic function~\eqref{eq:charStable}. We now claim that the position $S_n$ of the
strongly reinforced Shark Random Swim at time $n$ satisfies
\begin{equation}
\label{eq:connecSharkClusters}
S_n=_d\sum_{i=1}^n |c_{i,n}|\xi_i.
\end{equation}
Indeed, it readily follows from the described tree dynamics that the spin
attached to the node labeled $i$  corresponds to the $i$th step of the
shark: If $j>i$ and node $j$ is connected to node $i$ by an
intact edge, this means in terms of the shark that time $j$ is a memory
time, where the $i$th step is repeated. Node $j$ is then equipped with the
spin of its parent $i$. If, instead, the edge connecting $j$ to $i$ is cut,
this means that $j$ is a fresh time, and consequently, $j$ is equipped
with a new (independent) spin.

\begin{figure}
\begin{center}
\begin{tikzpicture}[scale = 1/3]
\node(U) at (-13,4) [circle,draw,inner sep=1pt] {\small $3$};
\node (Y) at (-15,0) [circle,draw,inner sep=1pt] {\small $1$};
\node (Z) at (-17,4) [circle,draw,inner sep=1pt] {\small $2$};
\node (A) at (1,0) [circle,draw,inner sep=1pt] {\small $1$};
\node(B) at (3,4) [circle, color=blue, draw,inner sep=1pt] {\small $3$};
\node (W) at (-19,8) [circle,draw,inner sep=1pt] {\small $4$};
\node (Q) at (-21,12) [circle,draw,inner sep=1pt] {\small $5$};
\node(C) at (-1,4) [circle,draw,inner sep=1pt] {\small $2$};
\node (D) at (-3,8) [circle,color=red,draw,inner sep=1pt] {\small $4$};
\draw (Y)--(Z);
\draw (U)--(Y);
\draw (W)--(Z);
\draw (A)--(C);
\draw (Q)--(W);
\draw (Y)--(-15,-1.5);
\draw (A)--(1,-1.5);
\draw (A)--(2,1.6);
\draw [color=blue] (B)--(2.3,2.4);
\draw (C)--(-2,5.6);
\draw [color=red] (D)--(-2.3,6.4);

\node (G) at (-1,0) { $\xi_1 $};
\node (H) at (-3,4) {$ \xi_1 $};
\node (I) at (5,4) [color=blue] {$ \xi_2 $};
\node (J)  at (-5,8) [color=red] {$ \xi_3 $};
\node (K) at (-8,12) [color=red] {$ \xi_3 $};
\node (F) at (-6,12) [circle,color=red,draw,inner sep=1pt] {\small $5$};
\draw [color=red] (D)--(F);
\node (L) at (-3.5,1.7) {$c_{1,5}$};
\node (M) at (5.5,1.7) [color=blue] {$c_{2,5}$};
\node (N) at (-8.5,9.7) [color=red] {$c_{3,5}$};
\end{tikzpicture}
\end{center}
\caption{On the left: An instance of $\mT_5$. On the right: The same tree
  after superposing percolation on $\mT_5$, i.e., an instance of
  $\mT_5^{(p)}$. Here, the edges between the vertices $1$, $3$ and $2$, $4$
  are cut. Consequently, three percolation clusters
  $c_{1,5},c_{2,5},c_{3,5}$ of sizes $|c_{1,5}|=|c_{3,5}|=2$, $|c_{2,5}|=1$
  arise. They are equipped with the spins $\xi_1$, $\xi_2$ and
  $\xi_3$. Since there are only three clusters, we let $c_{i,5}=\emptyset$
  for $i\geq 4$.}
\label{fig:ClustersSpins}
\end{figure}
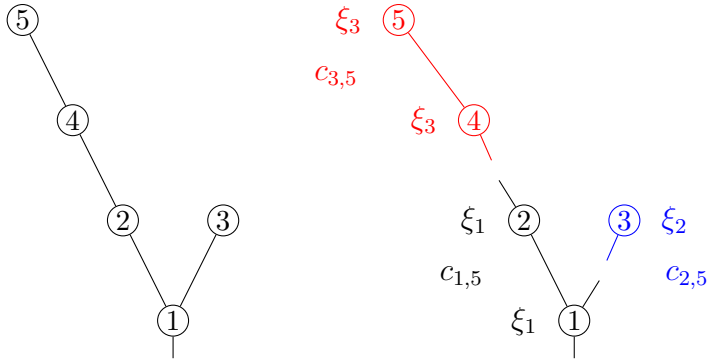

For what follows it is crucial
to notice that it makes no difference if we first build the tree $\mT_n$ and
then superpose percolation to obtain $\mT^{(p)}_n$, or if we dynamically
decide for each new vertex $i$ if the edge connecting $i$ to its parent
will be kept intact (with probability $p$), or cut at its midpoint (with
probability $1-p$).
\begin{remark}
  \label{rem:mere-reinforcedSRS}
  Clearly, in a similar way one could model the mere reinforced
  SRS. However, recall that the latter has only a partial memory, leading
  at each time $n$ to a weight increase with probability $p$ only. In
  particular, when building the tree with superposed percolation
  corresponding to the mere reinforced SRS, the weight of the parent node
  of a newly inserted edge may only be increased by the amount $b$ if the
  edge is kept intact. Therefore, in this case, one cannot \textit{first}
  build the tree and \textit{then} superpose percolation. Our
  techniques seem therefore less adequate to discuss the mere reinforced SRS.
\end{remark}

In order to make use of~\eqref{eq:connecSharkClusters}, we first have to
gain information on the cluster sizes $|c_{i,n}|$ of the above preferential
attachment tree when $n\rightarrow\infty$. To this aim, we will first give
an alternative continuous-time description of the tree, which will then
allow us to use techniques from branching processes.  This is the content
of the following section.
\subsection{Preferential attachment trees in continuous time}
\label{sec:continuousPAT}
This section is based on ideas from Bertoin and Uribe Bravo~\cite{BeBr}.

We let grow the preferential attachment tree $\mT_n$ introduced in the last section
in continuous time as follows. We start from the root node $1$ (with a
half-edge attached to it) at time $0$. Then, assuming that a tree with
$n\geq 1$ vertices has been constructed, we equip each vertex
$i\in\{1,\ldots,n\}$ with an independent exponential clock $\rho_i$ with
parameter $b(d_{n}(i)-1)+1$, where we write now $d_{n}(i)$ for the degree
of vertex $i$ when there are $n$ vertices present. The first clock rings at
time $\min_{i\in\{1,\ldots,n\}}\rho_i$, and then the vertex labeled $n+1$
is attached to the vertex $v_n=\arg\min_{i=1,\ldots,n}\rho_i$. Since the
sum of the degrees in the preferential attachment tree with $n$ vertices is
$2(n-1)+1$ (the $+1$ coming from the half-edge attached to the root), a
simple calculation shows
\begin{equation}
\label{eq:paramMinRhos}
\min_{i=1,\ldots,n}\rho_i=_d\textup{Exp}\left(
b(n-1)+n\right),
\end{equation}
where Exp$(s)$ denotes the exponential distribution with parameter $s>0$.

We shall write $T(t)$ for the tree constructed in this way at time
$t\geq 0$. Define $\tau_n$ to be the first instance when there are $n$
vertices in the tree, i.e.,
\[\tau_n:=\inf\{t\geq 0: |T(t)|=n\}.\]
Recalling~\eqref{eq:kndn}, it follows from the above dynamics that
$T(\tau_n)$ is a version of the preferential attachment tree $\mT_n$
constructed in Section~\ref{sec:connectionPAT}.

The fact that the $(n+1)$st vertex arrives after an exponential waiting time of
parameter $ b(n-1)+n$ suggests to
consider the process
\[Y(t):=b(|T(t)|-1)+|T(t)|,\quad t\geq 0.\]

\begin{lemma}
\label{lem:Yprocess}  
The process $\left(Y(t),t\geq 0\right)$ is a pure birth branching process
starting from $Y(0)=1$, which has only jumps of size $b+1$, and with unit
birth rate per unit population size. Moreover, the process
$(\e^{-(b+1)t}Y(t),t\geq 0)$ is a square-integrable martingale, whose
terminal value $W$ is
Gamma$\left(\frac{1}{b+1}, \frac{1}{b+1}\right)$-distributed.
\end{lemma}
The proof can be found in Appendix~\ref{sec:appendix}. We note for later use that
\begin{equation}
\label{eq:def-tau-n}
Y(\tau_n)=b(n-1)+n.\end{equation}

We now superpose percolation on $T(t)$ as follows: We assign to each edge
$e_i$ connecting a vertex $i\geq 2$ to its parent an independent uniform
variable $U_i$. If $U_i>p$, we cut $e_i$ at its midpoint, if $U_i\leq p$ we
let $e_i$ intact. We obtain a combinatorial structure $T^{(p)}(t)$ with the
same set of vertices as $T(t)$, and the subset of edges $e_i$ of $T(t)$ for
which $U_i\leq p$, together with half-edges; more precisely, two half-edges
for each edge $e_i$ of $T(t)$ for which $U_i>p$. We agree that cutting
edges preserves the degrees of the vertices.

The subtrees of $T^{(p)}(t)$ which are spanned by vertices connected to
each other by a path of intact edges form what we call the \textit{percolation
clusters} of $T^{(p)}(t)$.  We write
\[T^{(p)}_1(t), T^{(p)}_2(t),T^{(p)}_3(t),\ldots\] for these subtrees
enumerated in the increasing order of their birth times, with the
convention that $T^{(p)}_{i+1}(t) = \emptyset$ if the number of edges that
has been cut up to time $t$ is less than $i$, for $i\geq 1$. (We stress
that the $T^{(p)}_i(t)$ are combinatorial structures formed by vertices,
edges \textit{and} half-edges and are therefore not subtrees in the strict
graph theoretic sense, but we stick to that wording.)

In particular, $T^{(p)}_1(t)$ is the subtree rooted at vertex $1$, and, more
generally, if $U_j$ is the $i$th variable among
$U_2,U_3,\ldots$ to be greater than $p$, then $T^{(p)}_{i+1}(t)$ is the
subtree of $T^{(p)}(t)$ rooted at node $j$. Note moreover that
\[\sum_{i=1}^\infty|T^{(p)}_i(t)|=|T(t)|,\]
where $|T^{(p)}_i(t)|$ denotes the number of vertices of $T^{(p)}_i(t)$.

From the construction, we readily obtain the following connection to the clusters $c_{i,n}$
of percolation with parameter $p$ on the preferential attachment tree $\mT_n$:
\begin{cor}
\label{cor:linkDiskCont} There is the equality in distribution
\begin{equation*}
\left(|c_{1,n}|,|c_{2,n}|,\ldots\right) =_d
\left(|T^{(p)}_1(\tau_n)|,|T^{(p)}_2(\tau_n)|,\ldots\right).\end{equation*}
\end{cor}

It will be useful to introduce a notation for the birth time of
the $i$th subtree $T^{(p)}_{i}$. We set
\begin{equation}
\label{eq:birthtimes}
b_1:=0\quad\textup{and}\quad b_i:=\inf\{t\geq 0: T^{(p)}_i(t)\neq
\emptyset\},\quad i\geq 2.\end{equation}
Of course we should rather write $b_i^{(p)}$, but we skip the parameter $p$
for ease of reading. We warn, however, that $b_i$ should never be
confused with the reinforcement parameter $b$. 

We further denote by $H^{(p)}_i(t)$ the number of half-edges attached to
the vertices of $T^{(p)}_i(t)$. In particular, we have $H_i^{(p)}(b_i)=1$
and $H_i^{(p)}(t)=0$ for $0\leq t<b_i$. (Recall that $H_1^{(p)}(b_1)=1$
follows from the construction of the tree $T(t)$.)

It should be clear from the description that the processes
$(|T_i^{(p)}(b_i+t)|, t\geq 0)$ for $i\geq 1$ are independent and identically
distributed; in particular,
\begin{equation}
\label{eq:TiT1}
(|T_i^{(p)}(b_i+t)|, t\geq 0)=_d(|T_1^{(p)}(t)|, t\geq 0).
\end{equation}
Although we are primarily interested in the size processes
$|T_i^{(p)}(\cdot)|$, it is much more natural to look at
\begin{equation}
\label{eq:Y1T1H1}
Y_i^{(p)}(t):=\left(b\big(|T_i^{(p)}(t)|-2+H_i^{(p)}(t)\big)+|T_i^{(p)}(t)|\right)\1_{\{
  b_i\leq t\}}.
\end{equation}
Indeed, the processes $Y_i^{(p)}$ are easy to control, thanks to the following
lemma.
\begin{lemma}
\label{lem:Yi-processes}
  The processes
    $\big(Y_i^{(p)}(b_i+t),t\geq 0\big)$, $i\geq 1$, are
    i.i.d. pure birth branching processes starting from $Y_i^{(p)}(b_i)=1$, with
    unit birth rate per unit population size and reproduction law given by
    the law of $b+\eps_p$, where
    $\eps_p$ is Bernoulli-distributed with success probability $p$.
    Moreover, the following properties hold:
\begin{itemize}
\item $\E\left[Y_i^{(p)}(b_i+t)\right]=\e^{(b+p)t}\quad\textup{and}\quad
  \E\left[Y_i^{(p)}(b_i+t)^2\right]=\frac{(b+1)(b+2p)}{b+p}\left(\e^{2(b+p)t}-\e^{(b+p)t}\right).$
\item The process $\big(\e^{-(b+p)t}Y_i^{(p)}(b_i+t),t\geq 0\big)$ is a
martingale bounded in $L^k$ for any $k\in\N$, whose terminal value $W_i$ is almost surely
strictly positive, with
\[\E[W_i]=1\quad\textup{and}\quad \E\left[W_i^2\right]= \frac{(b+1)(b+2p)}{b+p}.\]
\end{itemize}
\end{lemma}

We stress that the variables $Y_i^{(p)}(t)$ are linked to $Y(t)$ via
$\sum_{i=1}^\infty Y_i^{(p)}(t) =Y(t)$.

\begin{remark}
  Analogously, one sees that the martingale
  $(\e^{-(b+1)t}Y(t),t\geq 0)$ from Lemma~\ref{lem:Yprocess} is bounded
  in $L^k$ as well, for any $k\in\N$ (and not merely in $L^2$); however,
  for our purpose, square-integrability will be sufficient.
\end{remark}  
  
We will close this section with an upper and lower bound on the birth times
$b_i$ defined in~\eqref{eq:birthtimes}. The following lemma extends~\cite[Lemma 8]{Bu1}.
\begin{lemma}
  \label{lem:birthtimes}
  Let $(x_n,n\in\N)$ be a sequence of positive integers with
  $\lim_{n\rightarrow\infty}x_n=\infty$ and $x_n\leq n$. Then there exists a sequence
  $(\vareps_n,n\in\N)$ of positive reals with $\vareps_n\downarrow 0$ as
  $n\rightarrow\infty$ and a sequence of events $(E_n,n\in\N)$ with
  $\lim_{n\rightarrow\infty}\P(E_n)=1$, such that
  on $E_n$, the following bounds hold for the birth times $b_i$ with
  $x_n\leq i\leq n$, provided $n$ is large enough:
\begin{align*}
  \tau_n-b_i&\,\,\leq\,\, t_{n,i}^+ :=\frac{1}{b+1}\left(\ln
              n-\ln(i-1)+\ln(1-p)+\vareps_n\right),\\ 
  \tau_n-b_i&\,\,\geq\,\, t_{n,i}^- :=\frac{1}{b+1}\left(\ln
              n-\ln(i+1)+\ln(1-p)-\vareps_n\right). 
\end{align*}
\end{lemma}
The proofs of Lemmas~\ref{lem:Yi-processes} and~\ref{lem:birthtimes} are
postponed to Appendix~\ref{sec:appendix}.

Let us point at a useful consequence of the above lemma. For $i,n\in\N$, define the random
variables
\begin{align*}
  \underline{X}_i(n)&:=|T_i^{(p)}\left(b_i+t_{n,i}^-\right)|,\\
  \overline{X}_i(n)&:=|T_i^{(p)}\left(b_i+t_{n,i}^+\right)|,
\end{align*}
with the convention that $|T_i^{(p)}(s)|:=0$ if $s<b_i$. It follows that on
the event $E_{n}$, for $n$ sufficiently large and $i\geq x_n$, we have 
\[\underline{X}_i(n)\leq |T_i^{(p)}(\tau_n)|\leq\overline{X}_i(n).\]
The obvious advantage of working with the variables $\underline{X}_i(n)$ is
that they are independent, and so are the variables $\overline{X}_i(n)$ (in
contrast to the variables $|T_i^{(p)}(\tau_n)|$, $i\geq 1$). Concerning
their laws, we have by~\eqref{eq:TiT1}, for $i\in\N$ fixed,
\[\left(\underline{X}_i(n),n\in\N\right)=_d\left(|T_1^{(p)}\left(t_{n,i}^-\right)|,n\in\N\right),\quad\left(\overline{X}_i(n), n\in\N\right)=_d\left(|T_1^{(p)}\left(t_{n,i}^+\right)|,n\in\N\right).\]

We will need to bound the moments of $\underline{X}_i(n)$ and
$\overline{X}_i(n)$ (in fact, a bound on the fourth moment will be
sufficient). Using that $\overline{X}_i(n)$ is stochastically dominated by
$Y_1^{(p)}(t_{n,i}^+)$, see~\eqref{eq:Y1T1H1}, the following corollary is
an immediate consequence of Lemma~\ref{lem:Yi-processes}.
\begin{cor}
\label{cor:momentsXi}
Let $p\in(0,1)$, $b\geq 0$, and $\kappa=(b+p)/(b+1)$. For each
$\ell\in\N$, there exists a constant $C_\ell=C_\ell(b,p)$ such that for all
$i,n\in\N$,
\[\E\left[\underline{X}_i(n)^\ell\right]\leq
  \E\left[\overline{X}_i(n)^\ell\right]\leq C_\ell\left(\frac{n}{i}\right)^{\ell\kappa}.\]
\end{cor}

\subsection{Cluster sizes of percolation on preferential attachment trees}
\label{sec:clustersizesPAT}
\subsubsection{Size of the root cluster}
In this section we study the size of the root cluster $c_{1,n}$ of
$\mT^{(p)}_n$. We work in the setting of Section~\ref{sec:continuousPAT},
with the variables and notation defined there. As always, we assume
$p\in(0,1)$, $b\geq 0$, and $\kappa=\kappa(b,p)= (b+p)/(b+1)$.

We recall that $|c_{1,n}|=_d |T_1^{(p)}(\tau_n)|$, the latter being linked to $Y_1^{(p)}(\tau_n)$
via~\eqref{eq:Y1T1H1}. We first establish a limit result
for $Y_1^{(p)}(\tau_n)$.
\begin{lemma}
  \label{lem:Y1conv}
  We have the convergence
\[ \lim_{n \rightarrow \infty} \frac{Y_1^{(p)}(\tau_n)}{n^\kappa} =
  \hZ_1\quad\textup{almost surely and in }L^2 , \]
where $\hZ_1$ is a strictly positive random variable with
\[\E[\hZ_1]=\frac{\Gamma\left(\frac{1}{b+1}\right)}{\Gamma\left(1+\frac{p}{b+1}\right)}\quad\textup{and}\quad
\E[\hZ^2_1]=\frac{(b+1)^2}{(b+p)}\frac{\Gamma\left(\frac{1}{b+1}\right)}{\Gamma\left(\frac{b+2p}{b+1}\right)}.
\]
\end{lemma}
\begin{remark} Unfortunately, we were not able to identify the law
of $\hZ_1$.  One can, however, be a bit more precise about $\hZ_1$, see
Display~\eqref{eq:ZW1W} in the proof below.
If $b=0$, $\kappa$ reduces to $p$ and $\hZ_1$ is known to follow the
Mittag-Leffler-distribution with parameter $p$, see~\cite[Lemma 3]{Bu1}.
\end{remark}

\begin{proof}
  All the following convergences hold almost surely and in
  $L^2$. By~\eqref{eq:def-tau-n} and Lemma~\ref{lem:Yprocess}, 
\[ \lim_{n \rightarrow \infty} \e^{-(b+1)\tau_n} \left(b(n-1)+n \right)=
  W. \] Since  
$(\e^{-(b+1)\tau_n})^{\kappa}=\e^{-(b+p)\tau_n}$,
Lemma~\ref{lem:Yi-processes} shows that (with $W_1$ from there)
\begin{equation}
\label{eq:ZW1W}
\lim_{n \rightarrow \infty} \frac{Y_1^{(p)}(\tau_n)}{n^\kappa}
= W_1\left(\frac{W}{b+1}\right)^{-\kappa}=:\hZ_1 . \end{equation}
It remains to argue why the first and second moment of $\hZ_1$
take the stated form. In this regard, we first note that $W$ is independent
of $\hZ_1$: Indeed, by definition $Y_1^{(p)}(\tau_n)$ depends only on the size of the subtree $T_1^{(p)}(t)$ and on the
number of half-edges $H_1^{(p)}(t)$ attached to it, evaluated at time $t=\tau_n$ when the $n$th
vertex arrives in the full tree $T(t)$. In particular, $Y_1^{(p)}(\tau_n)$
does not depend on the value of $\tau_n$, and therefore, $\hZ_1$ is independent
of $W$ (of course, $W$ and $W_1 $ are \textit{not} independent). Thus, for any $\gamma\geq 0$,
\[\E\left[\hZ_1^\gamma\right]=\frac{\E\left[W_1^\gamma\right]}{\E\left[\left(\frac{1}{b+1}W\right)^{\gamma\kappa}\right]}.\]
Specifying to $\gamma=1$ and $\gamma=2$, we obtain the claim from what we know about $W_1$ and $W$.
\end{proof}

In order to gain information about the root cluster, we need to control the
half-edges.

We recall from the introduction that $r_n\lesssim s_n$ for two sequences
$(r_n)$ and $(s_n)$ means $r_n\leq Cs_n$ for some $C>0$ independent of $n$.
\begin{lemma} 
\label{lem:halfedges}
For $b>0$ we have the
convergence in $L^2$
\[ \lim_{n \rightarrow \infty} \frac{H_1^{(p)}(\tau_n)-\frac{1-p}{b+p}Y_1^{(p)}(\tau_n)}{n^\kappa}=0.\]
\end{lemma}
\begin{proof}
First note that the two processes
\[H_1^{(p)}(t)-(1-p)\int_0^t Y_1^{(p)}(s)\dt s,\quad t\geq 0,\]
and
\[ Y^{(p)}_1(t)-(b+p)\int_0^t Y_1^{(p)}(s)\dt s,\quad t\geq 0,\]
are both martingales with respect to the natural filtration, and thus also the process
\[L(t):=H_1^{(p)}(t)-\frac{1-p}{b+p}Y_1^{(p)}(t),\quad t\geq 0,\] is a
(c\`adl\`ag) martingale. Writing $\lbrack L \rbrack_t$ for the quadratic
variation process of $L$, the process $L(t)^2- \lbrack L \rbrack_t$ is also
a martingale. In order to estimate the mean number of jumps of $L$, we
remark that the jump times of $L$ agree with those of $Y_1^{(p)}$. However,
in contrast to the jumps of $L$, the jumps of $Y_1^{(p)}$ are positive
only, namely of sizes either $b+1$ or $b$.  This implies that the mean
number of jumps of $Y_1^{(p)}$ (and hence of $L$) up to time $\tau_n$ is
bounded from above by
\[\frac{1}{b}\E\left[Y_1^{(p)}(\tau_n)\right]\lesssim n^\kappa,\]
where for the last estimate we have used Lemma~\ref{lem:Y1conv}. Using that the jump sizes of $L$ are bounded in absolute value by a
constant (depending on $b$), an application of the Burkholder-Davis-Gundy
inequality now shows that
\[\E\left[ (L(\tau_n)-L(0))^2 \right] \lesssim  \E\left[\lbrack L
  \rbrack_{\tau_n}\right] \lesssim n^\kappa. \] In particular,
\[\lim_{n\rightarrow\infty}\E\left[\left|\frac{L(\tau_n)}{n^\kappa}\right|^2\right]=0,\]
proving the lemma.
\end{proof}
We arrive at the following result for the size of the root cluster $c_{1,n}$.
\begin{prop}
 \label{prop:convRootCluster}
We have the convergence in $L^2$
\[\lim_{n \rightarrow \infty} \frac{|c_{1,n}|}{n^\kappa}  = Z_1,\]
where $Z_1=\frac{p}{b+p}\hZ_1$, with $\hZ_1$ as in Lemma~\ref{lem:Y1conv}.
\end{prop}
\begin{proof}
  By~\eqref{eq:Y1T1H1},
  \[|c_{1,n}|=_d\frac{Y_1^{(p)}(\tau_n)-bH_1^{(p)}(\tau_n)+2b}{b+1},\]
  and the claim follows from a combination of Lemmas~\ref{lem:Y1conv} and~\ref{lem:halfedges}.
\end{proof}

\subsubsection{Sizes of the remaining clusters}
\label{sec:sizesRemaining}
We now describe how we control the sizes of the percolation clusters rooted
at nodes different from the root. We will apply our results from this
section in the supercritical case $\alpha\kappa>1$.

We first look at a different but related quantity. Namely, let us write
$\mT_{i,n}^{(p)}$ for the subtree of $\mT_n^{(p)}$ rooted at node $i$,
i.e., $\mT_{i,n}^{(p)}$ is the combinatorial structure spanned by the
vertices $j\geq i$ which are connected to node $i$ after superposing
percolation on $\mT_n$. Note that $\mT_{1,n}^{(p)}$ equals the root cluster
$c_{1,n}$, whereas for $i\geq 2$, $\mT_{i,n}^{(p)}$ is a percolation
cluster of $\mT_n^{(p)}$ if and only if the edge connecting $i$ to its
parent has been cut. Let us also write $\eta(n,i)$ for the
number of nodes $j\geq i$ which are connected to node $i$ in $\mT_n$, that
is, \textit{before} superposing percolation. Since it does
not matter if we first build the tree and then perform percolation or
if we decide successively if a new arriving edge is cut or not, it holds that
\begin{equation}
\label{eq:Tin}
|\mT_{i,n}^{(p)}| =_d |c_{1, \eta(n,i)}|.
\end{equation}

The distribution of $\eta(n,i)$ can be modeled by means of a P\'olya urn
with diagonal (deterministic) replacement matrix
$\begin{psmallmatrix} b+1 & 0\\0&b+1\end{psmallmatrix}$. We stress that
$b\geq 0$ needs not to be an integer: Indeed, one might define the urn
process as a Markov process taking values in $\{(x,y)\in\R^2: x,y>0\}$ with
transitions from $(x,y)$ to $(x+b+1,y)$ with probability $x/(x+y)$, and
from $(x,y)$ to $(x,y+b+1)$ with probability $y/(x+y)$. However, for
simplicity, let us depict the connection to $\eta(n,i)$ as if we would add
$b+1$ balls at each step.

Initially at time one, the urn contains one green ball representing the
weight $k_i(i)$ of node $i$ and $(i-1)(b+1)$ red balls representing the sum
$\sum_{j=1}^{i-1}k_i(j)$ of weights of the nodes labeled $1,\ldots,i-1$,
just after the $i$th node has been inserted. We then draw repeatedly a ball
uniformly at random and return it together with $b+1$ balls of the same
color (corresponding to an additional weight increase of $b+1$, namely $b$
for the parent node and $1$ for the newly inserted node). The number
$G_{n-i}$ of green balls after $n-i$ draws (with $G_0=1$) corresponds to
the sum of weights of the vertices connected to node $i$ in $\mT_n$ (before
percolation), and we have
\begin{equation}
\label{eq:rel-etaGn}
\eta(n,i)=_d \frac{G_{n-i}-1}{b+1}+1=\frac{G_{n -i}+b}{b+1}. \end{equation}
We need to control $\eta(n,i)$ when $i$ is fixed and $n$ is large. First,
as an immediate consequence of~\cite[Corollary
3.1]{Ma}, we obtain for the first and second moment of $\eta(n,i)$ the bounds
\begin{equation}
\label{eq:etaFirstSecondMomentEstimate}
\E\left[\eta(n,i)\right]\lesssim\frac{n}{i},\quad\quad\
\E\left[\eta(n,i)^2\right]\lesssim\frac{n^2}{i^2}.
\end{equation}
Moreover, Theorem 3.2 of~\cite{Ma} shows that for fixed $i\in\N$, there is
the almost sure convergence
\begin{equation}
  \label{eq:limitEta}
  \lim_{n\rightarrow\infty}\frac{\eta(n,i)}{n}=\textup{Beta}\left(\frac{1}{b+1},i-1\right),\end{equation}
where $\textup{Beta}(r,s)$ denotes a Beta-distributed random variable with
parameters $r,s$ (with the convention $\textup{Beta}(r,0)=1$). Note,
however, that~\cite[Theorem 3.2]{Ma} is formulated for the number
$\tilde{G}_n$ of green ball \textit{drawings} after $n$ draws. But $G_n$ is
clearly related to $\tilde{G}_n$ via $G_n=\tilde{G}_n(b+1)+1$, and an
application of~\eqref{eq:rel-etaGn} yields~\eqref{eq:limitEta}.

Let us now come back to the percolation clusters $c_{1,n},c_{2,n},\ldots$
of $\mT_n^{(p)}$. We recall that for $i\geq 2$, $\mT_{i,n}^{(p)}$ is a
percolation cluster if and only if the edge between $i$
and its parent has been cut. Setting $\mathcal{C}_{i,n}:=\mT_{i,n}^{(p)}$
if the latter is a percolation cluster, and $\mathcal{C}_{i,n}:=\emptyset$
otherwise, we obtain:

\begin{cor}
  \label{cor:convci}
  Let $p\in(0,1)$, $b\geq 0$, $\kappa=(b+p)/(b+1)$, and let $i\in\N$. Then we have the $L^2$-convergence
  \[\lim_{n\rightarrow\infty}\frac{|\mathcal{C}_{i,n}|}{n^\kappa} =Z_i,\]
  where $Z_i$ is equal in distribution to
  \[\eps_i\cdot\beta_i^{\kappa}\cdot Z_1,\] with $\eps_1=1$ and $\eps_i$ for
  $i\geq 2$ a Bernoulli-distributed random variable with success
  probability $1-p$, $\beta_i$ a
  $\textup{Beta}\left(\frac{1}{b+1},i-1\right)$-distributed random
  variable independent of $\eps_i$, and $Z_1$ as in Proposition~\ref{prop:convRootCluster},
  independent of $\eps_i$ and $\beta_i$.
 \end{cor} 
 \begin{proof} For $i=1$, this is simply the statement of
   Proposition~\ref{prop:convRootCluster}. For general $i\in\N$, we have
   by~\eqref{eq:Tin} the equalities in distribution
   \begin{equation}
     \label{eq:Tin2}
       |\mathcal{C}_{i,n}|=_d\eps_i|\mT_{i,n}^{(p)}|=_d\eps_i|c_{1,\eta(n,i)}|.
     \end{equation}
   The claim then follows from~\eqref{eq:limitEta}, together with
   Proposition~\ref{prop:convRootCluster}.
\end{proof}
Clearly, unless $i=1$, we do not have $|c_{i,n}|=|\mathcal{C}_{i,n}|$ in
general. However, for our purposes, the much weaker
relation~\eqref{eq:SRSsumcixi} will be sufficient.

Finally, for future use we record that
   \begin{equation}
\label{eq:ZiMoments}
\sum_{i=1}^\infty \E\left[Z_i^\alpha\right]<\infty\quad\textup{if and
  only if}\quad \alpha\kappa>1.\end{equation}
Indeed, recall that for $q\geq 0$, the
$q$th moment of a Beta$(s,t)$-distributed random variable is given by
$B(s+q,t)/B(s+t)$, where $B(s,t)$ denotes the Beta-function with
parameters $s,t>0$.  A small calculation then leads to~\eqref{eq:ZiMoments}.

\subsection{Asymptotic behavior of the strongly reinforced SRS}
\label{sec:asympSRS}
We are now in position to discuss the long-time behavior of the strongly
reinforced Shark Random Swim $(S_n,n\in\N)$. We remind that this means that
$(S_n,n\in\N)$ is the strongly memory-reinforced random walk defined in
terms of an i.i.d. sequence $(\xi_i,i\in\N)$ of isotropic $\alpha$-stable random
variables with characteristic function~\eqref{eq:charStable}.

As we will show, the strongly reinforced SRS exhibits a phase transition at
$\alpha\kappa=1$. The stability parameter $\alpha$ may take values
in $(0,2]$, and $\kappa=\kappa(b,p)=(b+p)/(b+1)$.

If $\alpha\kappa\leq 1$, we prove weak convergence of finite-dimensional
laws towards an $\alpha$-stable stochastic process. In the subcritical case
$\alpha\kappa<1$, the limit process is non-L\'evy, whereas in the critical
case $\alpha\kappa=1$, it is an $\alpha$-stable L\'evy process. In the
supercritical case $\alpha\kappa>1$ we prove convergence in probability to
a (nontrivial) stochastic process.

In order to prove our results, we make use of the representation of
$(S_n,n\in\N)$ in terms of cluster sizes established
in~\eqref{eq:connecSharkClusters}. More precisely, recalling
Corollary~\ref{cor:linkDiskCont}, we shall work in the continuous setting
of Section~\ref{sec:continuousPAT} and define 
\[|c_{i,n}|:=|T_i^{(p)}(\tau_n)|\] 
to be the size of the $i$th percolation cluster of $T^{(p)}$ stopped at time $\tau_n$.

For fixed $i\in\N$, we may view $n\mapsto |c_{i,n}|$ as an
increasing $\N\cup\{0\}$-valued process. It readily follows
that~\eqref{eq:connecSharkClusters} can be strengthened to an equality in law for
processes, namely
\begin{equation}
\label{eq:connecSharkClusters2}
(S_n,n\in\N)=_d\left(\sum_{i=1}^n |c_{i,n}|\xi_i,n\in\N\right).\end{equation}
Since $c_{i,n}=\emptyset$ for $i>n$, we may as well sum up to infinity.

\subsubsection{The subcritical case $\alpha\kappa<1$}
Given $p\in(0,1)$ and $b\geq 0$, we define for $x>0$ the random
(almost surely c\`adl\`ag) function
\[f(x):=\left|T^{(p)}_1\left(\frac{1}{b+1}(-\ln x+\ln(1-p))\right)\right|,\]
recalling that we set $|T^{(p)}_1(s)|=0$ if $s<0$.

Bounding $|T^{(p)}_1(s)|$ from above by $Y^{(p)}_1(s)$,
Lemma~\ref{lem:Yi-processes} shows that
\begin{equation}
\label{eq:subcritical-intfinite}
\int_0^\infty\E\left[f(x)^\alpha\right]\dt
x=\int_0^{1-p}\E\left[f(x)^\alpha\right]\dt x\,\,\lesssim\,\,
\int_0^{1-p}x^{-\alpha\kappa}\dt x<\infty\quad\quad\textup{if }\alpha\kappa<1.\end{equation}

We let $\mathcal{X}=(\mathcal{X}_t, t\geq 0)$ denote a $d$-dimensional symmetric $\alpha$-stable stochastic
process with $\mathcal{X}_0:=0$, whose marginals
$(\mathcal{X}(t_1),\ldots,\mathcal{X}(t_k))$ for $0<t_1<t_2<
\dots< t_k$ have
characteristic function
\begin{equation}
  \label{eq:subcritical-processX}
  \E\left[\exp\left(i\sum_{j=1}^k\mathcal{X}(t_j)\cdot\theta_j\right)\right]=  \exp\bigg(-\int_0^\infty\E\left[\bigg\|\sum_{j=1}^kf(x/t_j)\theta_j\bigg\|^\alpha\right]\dt
    x\bigg),\quad\theta_1,\ldots,\theta_k\in\R^d.
\end{equation}
The existence of such a process in the subcritical case
$\alpha\kappa<1$ follows from Kolmogorov's existence theorem. We refer to
Chapter 3 of Samorodnitsky and Taqqu~\cite{SaTa} for more information on
stable processes and for a proof of the existence (given, however, in a
much more general setting).

\begin{thm}
\label{thm:SRSsubcritical}
Let $p\in(0,1)$, $b\geq 0$ and $\kappa=(b+p)/(b+1)$. Assume
$0<\alpha<1/\kappa$. Then, as $n\rightarrow\infty$, the finite-dimensional
marginals of the process 
  \[ \left(\frac{S_{\lfloor tn \rfloor}}{n^{1/\alpha}},t\geq 0 \right)\]
converge in law to those of an $\alpha$-stable stochastic process
$\mathcal{X}=(\mathcal{X}(t),t\geq 0)$ specified
by~\eqref{eq:subcritical-processX}.

\end{thm}
\begin{proof}
  For the ease of reading, we restrict ourselves to the two-dimensional
  marginals. The general case of $k$-dimensional marginals works in the
  same way, but is heavier in notation. We fix $t_2> t_1>0$ and $a_1$,
  $a_2\in\R$. We use the abbreviations $n_1:=\lfloor t_1n\rfloor$,
  $n_2:=\lfloor t_2n\rfloor$, interpreting $n_1$, $n_2$ as functions
  in $n$.  Let $F:\R^d\rightarrow\R$ be a
  bounded continuous function. By first conditioning on
  $|c_{1,n_1}|,\ldots,|c_{n_1,n_1}|$,
  $|c_{1,n_2}|,\ldots,|c_{n_2,n_2}|$ and then integrating out, we obtain,
  with $\xi=_d\xi_1$,
\[\E\left[F\bigg(\frac{1}{n^{1/\alpha}}\left(a_1S_{n_1}+a_2S_{n_2}\right)\bigg)\right] =
  \E\left[F\bigg(\frac{1}{n^{1/\alpha}}\bigg(\sum_{i=1}^{n_2}\left|a_1|c_{i,n_1}|+a_2|c_{i,n_2}|\right|^\alpha\bigg)^{1/\alpha}\xi\bigg)\right].\]
From the last two displays and standard properties of symmetric stable
random variables (see, e.g., Theorem 3.1.2 of ~\cite{SaTa}), the claim
follows if we show the convergence in probability
\begin{equation}
\label{eq:SRSsubcritical-toshow}
\lim_{n\rightarrow\infty}\frac{1}{n}\sum_{i=1}^{n_2}\left|a_1|c_{i,n_1}|+a_2|c_{i,n_2}|\right|^\alpha=\int_0^\infty\E\left[\left|a_1f(x/t_1)+a_2f(x/t_2)\right|^\alpha\right]\dt
x.\end{equation}
We split the sum in the last display into
\begin{equation}
\label{eq:SRSsubcritical-bound}
\sum_{i=1}^{n_2}\left|a_1|c_{i,n_1}|+a_2|c_{i,n_2}|\right|^\alpha
 =\sum_{i=1}^{\lfloor \ln n\rfloor-1}\left|a_1|c_{i,n_1}|+a_2|c_{i,n_2}|\right|^\alpha
  +\sum_{i=\lfloor \ln
    n\rfloor}^{n_2}\left|a_1|c_{i,n_1}|+a_2|c_{i,n_2}|\right|^\alpha.
\end{equation}
The first sum we bound by
\[
  \sum_{i=1}^{\lfloor \ln n\rfloor-1}\left|a_1|c_{i,n_1}|+a_2|c_{i,n_2}|\right|^\alpha\leq
  (|a_1|+|a_2|)^\alpha\sum_{i=1}^{\lfloor \ln n\rfloor}|c_{i,n_2}|^\alpha.
\]
Clearly, $|c_{i,n_2}|$ is stochastically dominated by
$|c_{1,n_2}|$, and therefore, by
Proposition~\ref{prop:convRootCluster},
\[\frac{1}{n}\sum_{i=1}^{\lfloor\ln n\rfloor}\E\left[|c_{i,n_2}|^\alpha\right]\lesssim
  \frac{\ln n}{n}\E\left[|c_{1,n_2}|^\alpha\right]\lesssim
  \ln n\, n^{\alpha\kappa-1}.\] Since $\alpha\kappa<1$, we deduce that
the first sum in~\eqref{eq:SRSsubcritical-bound} is
negligible. It remains to show that upon
dividing by $n$, the second sum in~\eqref{eq:SRSsubcritical-bound}
converges in probability to the integral in~\eqref{eq:SRSsubcritical-toshow}.

Set $x_n:=\lfloor \ln (n/t_2)\rfloor$. We let $(E_{n}, n\in\N)$ be a
sequence of events as specified in Lemma~\ref{lem:birthtimes}. Then
$\P(E_{n_1}\cap E_{n_2})\rightarrow 1$ as $n\rightarrow\infty$. On the
event $E_{n_1}\cap E_{n_2}$, we have for $i\geq \lfloor \ln n\rfloor$ the
lower and upper bounds
\[\underline{X}_i(n_1)\leq |c_{i,n_1}| \leq
  \overline{X}_i(n_1),\quad\quad \underline{X}_i(n_2)\leq
  |c_{i,n_2}| \leq \overline{X}_i(n_2),\]
where we recall that $\underline{X}_i(n)=|T^{(p)}_{i}(b_i+t_{i,n}^-)|$ and 
$\overline{X}_i(n)=|T^{(p)}_{i}(b_i+t_{i,n}^+)|$, with
\begin{align*}
t_{n,i}^- &=\frac{1}{b+1}\left(\ln
  n-\ln(i+1)+\ln(1-p)-\vareps_n\right),\\
  t_{n,i}^+& =\frac{1}{b+1}\left(\ln
             n-\ln(i-1)+\ln(1-p)+\vareps_n\right).
\end{align*}
Note that for two sequences $r=(r_n)$, $s=(s_n)$ of reals, we have, with
$\|r\|_\alpha:=(\sum_n |r_n|^\alpha)^{1/\alpha}$ denoting the $L^\alpha$-norm,
\[\|r\|_\alpha-\|s\|_\alpha\leq \|r+s\|_\alpha\leq \|r\|_\alpha+\|s\|_\alpha.\]
From this, it is readily
seen that our claim~\eqref{eq:SRSsubcritical-toshow} follows if we prove the convergences in probability
\begin{equation}
 \label{eq:SRSsubcritical-toshow1}
  \lim_{n\rightarrow\infty}\frac{1}{n}\sum_{i=\lfloor \ln
  n\rfloor}^{n_2}\left|a_1\overline{X}_{i}(n_1)+a_2\overline{X}_{i}(n_2)\right|^\alpha
  =\int_0^\infty\E\left[\left|a_1f(x/t_1)+a_2f(x/t_2)\right|^\alpha\right]\dt x\end{equation}
and
\begin{equation}
 \label{eq:SRSsubcritical-toshow2}  
\lim_{n\rightarrow\infty}\frac{1}{n}\sum_{i=\lfloor \ln
  n\rfloor}^{n_2}\left|\overline{X}_{i}(n_1)-\underline{X}_{i}(n_1)\right|^\alpha =\lim_{n\rightarrow\infty}\frac{1}{n}\sum_{i=\lfloor \ln
  n\rfloor}^{n_2}\left|\overline{X}_{i}(n_2)-\underline{X}_{i}(n_2)\right|^\alpha
=0.
\end{equation}
We first show convergence of the expectations. To that end, recall the
(random) function $f$ defined at the beginning of this section. As for the
expectation of the sum in~\eqref{eq:SRSsubcritical-toshow1}, we have
\begin{align*}\frac{1}{n}\sum_{i=\lfloor \ln
  n\rfloor}^{n_2}\E\left[\left|a_1\overline{X}_{i}(n_1)+a_2\overline{X}_{i}(n_2)\right|^\alpha\right]&=
  \frac{1}{n}\sum_{i=\lfloor \ln
  n\rfloor}^{n_2}\E\left[\left|a_1|T_1^{(p)}(t_{n_1,i}^+)|+a_2|T_1^{(p)}(t_{n_2,i}^+)|\right|^\alpha\right]\\
  &=\frac{1}{n}\sum_{i=\lfloor \ln
  n\rfloor}^{n_2}\E\left[\left|a_1f\left(\e^{-\vareps_{n_1}}\frac{i-1}{n_1}\right)+a_2f\left(\e^{-\vareps_{n_2}}\frac{i-1}{n_2}\right)\right|^\alpha\right].
  \end{align*}
  The sum on the right is a classical Riemann sum and converges upon
  letting $n\rightarrow\infty$ towards
\[\int_0^\infty\E\left[\left|a_1f(x/t_1)+a_2f(x/t_2)\right|^\alpha\right]\dt x.\]
Analogously,
\[\frac{1}{n}\sum_{i=\lfloor \ln
  n\rfloor}^{n_2}\E\left[\left|\overline{X}_{i}(n_1)-\underline{X}_{i}(n_1)\right|^\alpha\right] =\frac{1}{n}\sum_{i=\lfloor \ln
  n\rfloor}^{n_2}\E\left[\left|f\left(\e^{-\vareps_{n_1}}\frac{i-1}{n_1}\right)-f\left(\e^{+\vareps_{n_1}}\frac{i+1}{n_1}\right)\right|^\alpha\right],\]
and the right hand side converges to zero as $n\rightarrow\infty$. The
second sum in~\eqref{eq:SRSsubcritical-toshow2} is handled in the same way.

Using independence, the fact that $2\alpha\leq 4$ and
Corollary~\ref{cor:momentsXi} for the last step, the variance of $\sum_{i=\lfloor \ln
    n\rfloor}^{n_2}|a_1\overline{X}_i(n_1)+a_2\overline{X}_i(n_2)|^\alpha$ is upper
bounded by
\begin{equation}\label{eq:SRSsubcritical-var}
  \sum_{i=\lfloor \ln n\rfloor}^{n_2}\textup{Var}\left(\left|a_1\overline{X}_i(n_1)+a_2\overline{X}_i(n_2)\right|^\alpha\right)\lesssim
  \sum_{i=\lfloor \ln n\rfloor}^{n_2}\E\left[\overline{X}_i(n_2)^4\right]^{2\alpha/4}\lesssim \sum_{i=\lfloor \ln n\rfloor}^{n_2}\left(\frac{n}{i}\right)^{2\alpha\kappa}.\end{equation}
Since $\alpha\kappa<1$, the right hand side
of~\eqref{eq:SRSsubcritical-var} is of order $o(n^2)$, as
wanted. Altogether, this proves~\eqref{eq:SRSsubcritical-toshow} and hence
the theorem.
\end{proof}

\subsubsection{The critical case $\alpha=1/\kappa$}
We recall the definition of the random variable $Z_1$ from
Proposition~\ref{prop:convRootCluster}. 

We let $\mathcal{X}=(\mathcal{X}_t, t\geq 0)$ denote a $d$-dimensional
symmetric $\alpha$-stable stochastic process with $\mathcal{X}_0:=0$, whose
marginals $(\mathcal{X}(t_1),\ldots,\mathcal{X}(t_k))$ for
$0=:t_0<t_1<t_2< \dots < t_k$ have characteristic function
\begin{equation}
  \label{eq:critical-processX}
\E\left[\exp\left(i\sum_{j=1}^k\mathcal{X}(t_j)\cdot\theta_j\right)\right]=  \exp\left(-\frac{1-p}{b+1}\E[Z_1^\alpha]\sum_{j=1}^k(t_j-t_{j-1})\bigg\|\sum_{i=j}^k\theta_i\bigg\|^\alpha\right),\quad\theta_1,\ldots,\theta_k\in\R^d.
\end{equation}
We recognize the characteristic function of the marginals of a
$d$-dimensional $\alpha$-stable L\'evy process, scaled by the factor
$(\frac{1-p}{b+1}\E[Z_1^\alpha])^{1/\alpha}$.

\begin{thm}
  \label{thm:SRScritical}
  Let $p\in(0,1)$, $b\geq 0$ and $\kappa=(b+p)/(b+1)$. Assume
  $\alpha=1/\kappa$. Then, as $n\rightarrow\infty$, the finite-dimensional
marginals of the process 
  \[ \left(\frac{S_{\lfloor n^t \rfloor}}{\left(n^t\ln
          n\right)^\kappa},t\geq 0 \right)\]
converge in law to those of an $\alpha$-stable L\'evy process
$\mathcal{X}=(\mathcal{X}(t),t\geq 0)$ specified
by~\eqref{eq:critical-processX}.
\end{thm}
\begin{proof}
  The proof is in spirit similar to the proof of the subcritical case. For
  ease of reading, we look again at the two-dimensional marginals only. We
  fix $t_2> t_1>0$ and set $n_1=n_1(n):=\lfloor n^{t_1}\rfloor$,
  $n_2=n_2(n):=\lfloor n^{t_2}\rfloor$.  Let $F:\R^d\times\R^d\rightarrow\R$ be a
  bounded continuous function. Then, by conditioning on the cluster sizes and
  integrating,
  \begin{align*}\lefteqn{\E\left[F\bigg(\frac{S_{n_1}}{(n_1\ln
    n)^{\kappa}},\frac{S_{n_2}}{(n_2\ln
    n)^{\kappa}}\bigg)\right]}\\
    &=
      \E\left[F\bigg(\frac{1}{(n_1\ln
      n)^\kappa}\bigg(\sum_{i=1}^{n_1}|c_{i,n_1}|^\alpha\bigg)^{\frac{1}{\alpha}}\xi,\frac{1}{(n_2\ln
      n)^\kappa}\bigg(\bigg(\sum_{i=1}^{n_1}|c_{i,n_2}|^\alpha\bigg)^{\frac{1}{\alpha}}\xi+\bigg(\sum_{i=n_1+1}^{n_2}|c_{i,n_2}|^\alpha\bigg)^{\frac{1}{\alpha}}\xi'\bigg)\bigg)\right],
  \end{align*}
where $\xi$ and $\xi'$ are i.i.d. copies of $\xi_1$. The stated joint weak convergence of the
marginals is now a consequence of the following convergences in
probability (recall that $\alpha\kappa=1$):
\begin{equation*}
  \lim_{n\rightarrow\infty}\frac{\sum_{i=1}^{n_1}|c_{i,n_1}|^\alpha}{n_1\ln
    n}=t_1\frac{1-p}{b+1}\E[Z_1^\alpha],\quad
  \lim_{n\rightarrow\infty}\frac{\sum_{i=1}^{n_1}|c_{i,n_2}|^\alpha}{n_2\ln n}=t_1\frac{1-p}{b+1}\E[Z_1^\alpha]\end{equation*}
and
\begin{equation}
\label{eq:SRScritical-toshow2}
\lim_{n\rightarrow\infty}\frac{\sum_{i=n_1+1}^{n_2}|c_{i,n_2}|^\alpha}{n_2\ln n}=(t_2-t_1)\frac{1-p}{b+1}\E[Z_1^\alpha].
\end{equation}
We show only~\eqref{eq:SRScritical-toshow2}, the first two convergences
follow from identical arguments.

We work again on the events $E_{n}$ from Lemma~\ref{lem:birthtimes},
defined with respect to $x_{n}:=\lfloor \ln n\rfloor$. For large $n$, we estimate
\begin{equation}
 \label{eq:SRScritical-lowerbound} 
\sum_{i=n_1+1}^{n_2}|c_{i,n_2}|^\alpha\1_{E_{n_2}}\geq
\sum_{i=n_1+1}^{\lfloor n_2/x_n\rfloor}\underline{X}_i(n_2)^\alpha\1_{E_{n_2}},\end{equation}
and 
\begin{equation}
\label{eq:SRScritical-upperbound}
\sum_{i=n_1+1}^{n_2}|c_{i,n_2}|^\alpha\1_{E_{n_2}}\leq
\sum_{i=n_1}^{\lfloor n_2/x_n\rfloor}\overline{X}_i(n_2)^\alpha+\sum_{i=\lfloor n_2/x_n\rfloor}^{n_2}\overline{X}_i(n_2)^\alpha.
\end{equation} 
Display~\eqref{eq:SRScritical-toshow2} follows if we show that both right hand sides
of~\eqref{eq:SRScritical-lowerbound} and~\eqref{eq:SRScritical-upperbound}
converge in probability upon dividing by $n_2\ln n$ to
\[(t_2-t_1)\frac{1-p}{b+1}\E[Z_1^\alpha].\] 

We restrict ourselves to the upper bound~\eqref{eq:SRScritical-upperbound}, the convergence of the lower
bound~\eqref{eq:SRScritical-lowerbound} can be shown in the same way. First, by 
Corollary~\ref{cor:momentsXi} and the fact that $\alpha\kappa=1$,
\[\sum_{i=\lfloor n_2/x_n\rfloor}^{n_2}\E\left[\overline{X}_i(n_2)^\alpha\right]\,\lesssim \,\sum_{i=\lfloor n_2/x_n\rfloor}^{n_2}\frac{n_2}{i},\]
and the right hand side is of order $o(n_2\ln n)$ as $n\rightarrow\infty$, by the choice of $x_n$. It
remains to prove the convergence in probability
\begin{equation}
\label{eq:SRScritical-Xi}
\lim_{n\rightarrow\infty}\frac{1}{n_2\ln
  n}\sum_{i=n_1}^{\lfloor n_2/x_n\rfloor}\overline{X}_i(n_2)^\alpha=(t_2-t_1)\frac{1-p}{b+1}\E[Z_1^\alpha].\end{equation}
We use again the second moment method. As far as the convergence of the
expectations is concerned, we notice that for $i=i(n)$ between $n_1$ and
$\lfloor n_2/x_n\rfloor$, we have $t_{n_2,i}^+\geq c\ln\ln n$ for some constant $c>0$. Arguments entirely similar to those leading to
Proposition~\ref{prop:convRootCluster} then show that in $L^2$ (and hence
in $L^\alpha$), uniformly in $i$ with $n_1\leq i\leq \lfloor n_2/x_n\rfloor$,
\begin{equation}
\label{eq:SRScritical-uniformLalpha}
  \lim_{n\rightarrow\infty}\frac{|T_1^{(p)}(t_{n_2,i}^+)|}{(n_2/i)^\kappa}
  =\left(\frac{1-p}{b+1}\right)^\kappa Z_1.\end{equation}
Since
\[\lim_{n\rightarrow\infty}\frac{1}{\ln n}\sum_{i=n_1}^{\lfloor n_2/x_n\rfloor}\frac{1}{i} = t_2-t_1,\]
we deduce from the uniform $L^\alpha$-convergence
in~\eqref{eq:SRScritical-uniformLalpha} that
\[\frac{1}{n_2\ln n}\sum_{i=n_1}^{\lfloor n_2/x_n\rfloor}\E\left[\overline{X}_i(n_2)^\alpha\right]=(t_2-t_1)\frac{1-p}{b+1}\E\left[Z_1^\alpha\right].\]
It remains
to show that the variance of the sum
$\sum_{i=n_1}^{\lfloor n_2/x_n\rfloor}\overline{X}_i(n_2)^\alpha$ is of order $o(n_2^2\ln^2
n)$. Using independence, we obtain with Corollary~\ref{cor:momentsXi}
\begin{equation}\label{eq:SRScritical-var}
  \textup{Var}\bigg(\sum_{i=n_1}^{\lfloor n_2/x_n\rfloor}\overline{X}_i(n_2)^\alpha\bigg)=\sum_{i=n_1}^{\lfloor n_2/x_n\rfloor}\textup{Var}\left(\overline{X}_i(n_2)^\alpha\right)\lesssim
  \sum_{i=n_1}^{n_2}\E\left[\overline{X}_i(n_2)^4\right]^{2\alpha/4}\lesssim\sum_{i=n_1}^{n_2}\left(\frac{n_2}{i}\right)^2,\end{equation}
and the right hand side
of~\eqref{eq:SRScritical-var} is in fact of order $o(n_2^2)$ as $n\rightarrow\infty$.
\end{proof}

\begin{remark} If in the critical case $\alpha=2$, we deduce from~\eqref{eq:critical-processX}
  that the limiting process $\mathcal{X}$ is a Brownian motion scaled by
  the factor
  \[\sqrt{2\frac{1-p}{b+1}\E\left[Z_1^2\right]}.\]
  The choice $\alpha=2$ implies $b=1-2p$ (which is,
  of course, only possible if $p\leq 1/2$). From what we know about $Z_1$,
  we see that the above scaling factor simplifies to
  $\sqrt{\frac{4p^2}{1-p}}$. Up to a factor $\sqrt{2}$, this is exactly what
  we find in the critical case for the strongly reinforced ERW, see
  Theorem~\ref{thm:ERW2} $b)$. The factor $\sqrt{2}$ comes from the fact that if
  $\alpha=2$, the steps of the shark are normally distributed with variance
  $2$ (and not $1$).
  \end{remark}

\subsubsection{The supercritical case $\alpha\kappa>1$}
We recall from Corollary~\ref{cor:convci} that the random variables $(Z_i,i\in\N)$ are defined as the $L^2$-limits of
$|\mathcal{C}_{i,n}|/n^{\kappa}$ upon letting $n\rightarrow\infty$, where
$|\mathcal{C}_{i,n}|$ denotes the size of the percolation cluster of
$\mT_n^{(p)}$ rooted at node $i$, with $|\mathcal{C}_{i,n}|=0$ if there is
no such cluster.  Clearly, it holds that
\begin{equation}\label{eq:SRSsumcixi}
  \sum_{i=1}^{n}|c_{i,n}|\xi_i =_d\sum_{i=1}^{n}|\mathcal{C}_{i,n}|\xi_i,\end{equation}
and for proving the following theorem, we will define $S_n$ via the right hand side
of~\eqref{eq:SRSsumcixi}, i.e., we set
$S_n:=\sum_{i=1}^{n}|\mathcal{C}_{i,n}|\xi_i$. As we will see, the sum
\begin{equation}
\label{eq:defZ}
Z:=\sum_{i=1}^\infty Z_i\xi_i
\end{equation}
is well-defined and appears in the limit of the strongly reinforced SRS defined in the above way.  
\begin{thm}
\label{thm:SRSsupercritical}
Let $p\in(0,1)$, $b\geq 0$, $\kappa=(b+p)/(b+1)$, and $t\geq 0$. Let
$\alpha$ satisfy $1/\kappa<\alpha\leq 2$, and let $Z$ be given
by~\eqref{eq:defZ}. Then $|Z|<\infty$ almost surely, and we have the
convergence in probability
\[\lim_{n\rightarrow\infty}\frac{S_{\lfloor
      tn\rfloor}}{n^{\kappa}}=t^\kappa Z.\]
\end{thm}
\begin{proof} 
  Obviously, we may suppose that $t=1$. Since conditionally on
  $Z_1,\ldots,Z_n$,
  $\sum_{i=1}^nZ_i\xi_i =_d (Z_1^\alpha
  +\ldots+Z_n^\alpha)^{1/\alpha}\xi_1$, we deduce
  from~\eqref{eq:ZiMoments} that $|Z|<\infty$ almost surely. Now let
  $F:\R^d\rightarrow\R$ be a bounded continuous function. We have
  \begin{align*}\E\bigg[F\bigg(\frac{S_n}{n^{\kappa}}-Z\bigg)\bigg]
    &=\E\bigg[F\bigg(\sum_{i=1}^\infty\left(\frac{|\mathcal{C}_{i,n}|}{n^{\kappa}}-Z_i\right)\xi_i\bigg)\bigg]\\
    &=\E\bigg[F\bigg(\bigg(\sum_{i=1}^\infty\left|\frac{|\mathcal{C}_{i,n}|}{n^{\kappa}}-Z_i\right|^\alpha\bigg)^{1/\alpha}\xi\bigg)\bigg],
    \end{align*}
    where $\xi=_d\xi_1$, and the last equality follows from conditioning on
    $|\mathcal{C}_{1,n}|,\ldots,|\mathcal{C}_{n,n}|$ and on the sequence
    $(Z_i,i\in\N)$.  In particular, the theorem
    follows once we have established that
\begin{equation}
\label{eq:SRSsupercrit-toshow}
\lim_{n\rightarrow\infty}
\sum_{i=1}^\infty\E\left[\left|\frac{|\mathcal{C}_{i,n}|}{n^{\kappa}}-Z_i\right|^\alpha\right]=0.\end{equation}
To that aim, we first observe from Corollary~\ref{cor:convci} that for each fixed $i\geq 1$ and $0<\alpha\leq 2$,
\[\lim_{n\rightarrow\infty}\E\left[\left|\frac{|\mathcal{C}_{i,n}|}{n^{\kappa}}-Z_i\right|^\alpha\right]=0.\]
It only remains to argue why the order of sum and limit
in~\eqref{eq:SRSsupercrit-toshow} can be interchanged. In this regard, it
suffices to show that the sum in~\eqref{eq:SRSsupercrit-toshow} is bounded
uniformly in $n$. First, for $1\leq i\leq n$, we estimate
\[\E\left[\left|\frac{|\mathcal{C}_{i,n}|}{n^{\kappa}}-Z_i\right|^\alpha\right]\leq
  2^\alpha\left(\frac{\E\left[|\mathcal{C}_{i,n}|^\alpha\right]}{n^{\alpha\kappa}} +
    \E[Z_i^\alpha]\right).\] The sum over the terms
$\E[Z_i^\alpha]$ is bounded, see~\eqref{eq:ZiMoments}.

For the sum over the terms
$\E\left[|\mathcal{C}_{i,n}|^\alpha\right]/n^{\alpha\kappa}$, the equality in law~\eqref{eq:Tin2}
gives for $i\geq 2$
\[
  \E\left[|\mathcal{C}_{i,n}|^\alpha\right]=(1-p)\sum_{\ell=1}^{n-i+1}\E\left[|c_{1,\ell}|^\alpha\right]\P\left(\eta(n,i)=\ell\right).\]
For $\ell$ sufficiently large, say $\ell\geq \ell_0$, we have by Proposition~\ref{prop:convRootCluster}
$\E[|c_{1,\ell}|^\alpha]\leq
2\E[Z_1^\alpha]\ell^{\alpha\kappa}$. Thus, using the elementary bound $\E[|c_{1,\ell}|^\alpha]\leq \ell_0^\alpha$ for
$\ell=1,\ldots,\ell_0$, one finds  
\[\sum_{\ell=1}^{n-i+1}\E\left[|c_{1,\ell}|^\alpha\right]\P\left(\eta(n,i)=\ell\right)\lesssim
  \ell_0^{1+\alpha}+\E\left[\eta(n,i)^{\alpha\kappa}\right].\] Using $\alpha\kappa\leq 2$
and~\eqref{eq:etaFirstSecondMomentEstimate} for the second moment of
$\eta(n,i)$, we deduce that
\[\E\left[|\mathcal{C}_{i,n}|^\alpha\right]\lesssim\ell_0^{1+\alpha}+
  c_0\E\left[\eta(n,i)^2\right] ^{\alpha\kappa/2}\lesssim
  \ell_0^{1+\alpha}+\left(\frac{n}{i}\right)^{\alpha\kappa}.\]
Since $\alpha\kappa>1$, this implies that the sum
$\frac{1}{n^{\alpha\kappa}}\sum_{i=1}^n\E[|\mathcal{C}_{i,n}|^\alpha]$ is uniformly bounded in $n$, which
proves what was left to show.

\end{proof}

\section{Generalizations and perspectives}
\label{sec:GenPer}
When dealing with the (strongly) reinforced ERW model in
Section~\ref{sec:ERW}, we restricted ourselves for simplicity reasons to
the one-dimensional case with symmetric $\pm 1$-steps. However, our
representation by means of an urn immediately generalizes to models in
higher dimensions $d\geq 2$. Even more generally, we could as well consider
a sequence $(\xi_i,i\in\N)$ of i.i.d. steps taking finitely many values
$v_1,\ldots,v_{k}\in\R^d$ according to some probability vector
${\bf p}=(p_1,\ldots,p_{k})^T$, i.e.,
\[\P\left(\xi_1=v_i\right)=p_i,\quad i=1,\ldots,k,\quad \textup{for
    some }k\in\N.\] Indeed, an urn with balls of $4k-1$ different colors
would be sufficient to model the position of the corresponding (strongly)
memory-reinforced walk, and upon appropriate normalization, the same limits
appear. What is, however, crucial is that the steps take only finitely many
values, since Janson's results~\cite{Ja} presuppose urns with a finite set
of types (see Remark 4.1 in the cited paper).

The finiteness restriction does not appear in the way we deal with the Shark
Random Swim, and therefore, this approach looks promising to be more widely
applicable. Indeed, it is natural to ask whether our limit results for the 
strongly reinforced SRS do also hold for steps that are merely in the
domain of attraction of an $\alpha$-stable law, and whether the techniques
from Section~\ref{sec:SRS} can be extended in that direction. In
particular, this would result in a unifying statement, where the (strongly)
reinforced ERW would arise just as a special case of the (strongly)
reinforced SRS when $\alpha=2$. The fact that we recognize in the Gaussian
case $\alpha=2$ the behavior already observed in Theorem~\ref{thm:ERW2}
underlines the reasonableness of this issue.

It is, however, somehow more natural to first look at the original ERW and
SRS models, without additional reinforcement. In this regard, for
square-integrable step variables, Bertoin~\cite{Be3} proved recently a
universal behavior in the supercritical regime, and a version of Donsker's
invariance principle in the diffusive regime. As far as the Shark Random
Swim is concerned,~\cite{Be1} contains generalizations of~\cite{Bu1} to
infinitely divisible step distributions.

In a different direction, it would be of interest
to study the limit behavior the mere reinforced Shark Random Swim. Theorem~\ref{thm:ERW1} may hint at what to expect in the
case $\alpha=2$. As we explain in Remark~\ref{rem:mere-reinforcedSRS}, the
fact that the reinforcement by the amount of $b$ affects only steps which
are repeated makes the situation more complex to handle.

\begin{appendix}
\section{Appendix}
\label{sec:appendix}
\begin{proof}[Proof of Lemma~\ref{lem:Yprocess}]
The fact that $(Y(t),t\geq 0)$ is a pure birth process with the stated
properties is a consequence of~\eqref{eq:paramMinRhos} and of the dynamics
of $(T(t),t\geq 0)$. Standard properties of branching processes (see,
e.g.,~\cite{AtNe}) show that $(\e^{-(b+1)t}Y(t),t\geq 0)$ is a
square-integrable martingale, and it follows
  from Lemma 3 in~\cite{BeGo} that its (almost surely and $L^2$-)limit is 
  Gamma$(1/(b+1),1/(b+1))$-distributed.
\end{proof}

\begin{proof}[Proof of Lemma~\ref{lem:Yi-processes}]
The i.i.d. property of the processes $(Y_i^{(p)}(b_i+\cdot),t\geq 0)$,
$i\geq 1$, is obvious from the construction. We shall therefore prove
everything for $i=1$, in which case $b_i=b_1=0$.

Clearly, the sum of degrees of vertices of $T^{(p)}_1(t)$ is equal
to \[\left(2(|T_1^{(p)}(t)|-1)+H^{(p)}_1(t)\right).\] It now follows from
the construction of the preferential attachment tree $T(t)$ at the
beginning of Section~\ref{sec:continuousPAT} (recall in particular the
parameters of the exponential clocks) that $(Y_1^{(p)}(t),t\geq 0)$ is a
pure birth process with the stated birth rate and reproduction law. It is
then well-known (see again~\cite{AtNe}) that
$(\e^{-(b+p)t} Y_1^{(p)}(t),t\geq 0)$ is a martingale, whose terminal
value $W_1$ is almost surely strictly positive. By Kolmogorov's forward
equation (see once more~\cite{AtNe}) we compute for $t>0$
\[\E\left[Y_1^{(p)}(t)\right]=\e^{(b+p)t},\quad\quad
  \E\left[Y_1^{(p)}(t)^2\right]=\frac{(b+1)(b+2p)}{b+p}\left(\e^{2(b+p)t}-\e^{(b+p)(b_i+t)}\right).\]
This proves square-integrability of $(\e^{-(b+p)t} Y_1^{(p)}(t),t\geq
0)$, and the claim about the first and second moment of $W_1$ follows from
the last display.

It remains to show boundedness in $L^k$ for $k\geq 3$, that is, we have to
show that there exists a constant $c_k<\infty$ such that
\begin{equation}
  \label{eq:Yi-boundedLk}
  \E\left[Y_1^{(p)}(t)^k\right]\leq c_k \e^{k(b+p)t}\quad\textup{for all
  }t\geq 0.\end{equation}
In order to prove this, we adapt~\cite[Proof of Lemma 3]{Be2} to
our situation. First, we note that the generator $\mathfrak{G}$ of $(Y_1^{(p)}(t),t\geq
0)$ is given for any smooth function $f:(0,\infty)\rightarrow\R$ by
\[\mathfrak{G}f(x)=x(1-p)\left(f(x+b)-f(x)\right)+xp\left(f(x+b+1)-f(x)\right).\]
Specifying to $f(x)=x^\ell$ for some integer $\ell\geq 3,$
\begin{align}\nonumber
  \mathfrak{G}f(x)&=x(1-p)\sum_{j=0}^{\ell-1}{\ell\choose
                    j}x^jb^{\ell-j}+xp\sum_{j=0}^{\ell-1}
                    {\ell\choose j}x^j(b+1)^{\ell-j}\\
                  &=\ell(b+p)x^\ell+(1-p)\sum_{j=0}^{\ell-2}{\ell\choose
                    j}x^{j+1}b^{\ell-j}+p\sum_{j=0}^{\ell-2}{\ell\choose
                    j}x^{j+1}(b+1)^{\ell-j}.
\label{eq:generator}
\end{align}
We prove now by induction that~\eqref{eq:Yi-boundedLk} holds for all
$k\in\N$. We already know it for $k=1$ and $k=2$, so let us assume that for
some $\ell \geq 3$,~\eqref{eq:Yi-boundedLk} holds for all
$k=1,\ldots,\ell-1$. Kolmogorov's forward equation reads
\[\frac{\textup{d}}{\textup{d}t}\E\left[f(Y_1^{(p)}(t))\right]=\E\left[\mathfrak{G}f(Y_1^{(p)}(t))\right].\]
In combination with~\eqref{eq:generator}, and using~\eqref{eq:Yi-boundedLk}
for $k=1,\ldots,\ell-1$, we deduce that for some $\gamma>0$ depending on $b$, we have
\begin{equation}
\label{eq:generator2}
\frac{\textup{d}}{\textup{d}t}\ln\E\left[Y_1^{(p)}(t)^\ell\right]\leq(b+p)\ell+\gamma\frac{\e^{(\ell-1)(b+p)t}}{\E\left[Y_1^{(p)}(t)^\ell\right]}.\end{equation} 
By Jensen's inequality,
\[\E\left[Y_1^{(p)}(t)^\ell\right]\geq
  \e^{\ell(b+p)t}\quad\textup{for all }t\geq 0,\]
so that

\[\int_0^\infty
  \frac{\e^{(\ell-1)(b+p)t}}{\E\left[Y_1^{(p)}(t)^\ell\right]}\dt t\leq
    \int_0^\infty\e^{-(b+p)t}\dt t=\frac{1}{b+p}.\]
  Going back to~\eqref{eq:generator2} and integrating, we obtain
  \[\E\left[Y_1^{(p)}(t)^\ell\right]\leq
    \e^{\gamma\frac{1}{b+p}}\e^{\ell(b+p)t}\quad\textup{for all }t\geq 0.\]
Thus,~\eqref{eq:Yi-boundedLk} does hold for $k=\ell$ as well, as wanted.
\end{proof}
\begin{proof}[Proof of Lemma~\ref{lem:birthtimes}]
  We fix a small $\vareps>0$ and a sequence $(x_n,n\in\N)$ of positive
  integers with $\lim_{n\rightarrow\infty}x_n=\infty$ and $x_n\leq n$. Recalling
  Lemma~\ref{lem:Yprocess} and the notation from there, we define for each $k\in\N$ the event
  \[E^1_k:=\left\{W(1-\vareps)\leq
      \e^{-(b+1)\tau_k}\left((b+1)k-b\right)\leq W(1+\vareps)\right\}.\]
  Lemma~\ref{lem:Yprocess} ensures that
  $\lim_{n\rightarrow\infty}\P\left(\bigcap_{k=x_n}^\infty
    E^1_k\right)=1$. On
  $E^1_k$, it holds for $k$ sufficiently large that
  \begin{align}
\label{eq:tauk-bounds}
\nonumber
 \tau_k &\leq \frac{1}{b+1}\left(\ln k - \ln W
         +\ln(b+1)-\ln\left(1-\vareps\right)\right),\\
\tau_k &\geq \frac{1}{b+1}\left(\ln k - \ln W
         +\ln(b+1)-2\ln\left(1+\vareps\right)\right).
\end{align}
Writing $D(k)$ for the number of subtrees present at time $\tau_k$, i.e.,
\[D(k)=\max\left\{i\geq 1: T_i^{(p)}(\tau_k)\neq \emptyset\right\},\]
we deduce from the construction of $T^{(p)}(t)$ that $D(k)$ has the same
law as $1+\sum_{i=1}^{k-1}\eps_{i,1-p}$, where $\eps_{i,1-p}$,
$i\geq 1$, are i.i.d. Bernoulli random variables with success probability $1-p$.
Consequently, an application of the law of large numbers shows that if we
define
\[E^2_k:=\left\{k(1-p)(1-\vareps)\leq D(k) \leq
    k(1-p)(1+\vareps)\right\},\] then
$\lim_{n\rightarrow\infty}\P\left(\bigcap_{k=x_n}^\infty E_k^2\right)=1$. On
$E_k^2$ it holds by construction that
\[b_{\lceil k(1-p)(1+\vareps)\rceil}\geq \tau_k.\]
Using~\eqref{eq:tauk-bounds}, we find that on the event $E^1_k\cap
E^2_k$, for $k$ large enough and provided $\vareps$ is sufficiently small,
\[b_k \geq \tau_{\lfloor\frac{k}{(1-p)(1+\vareps)}\rfloor}\geq 
  \frac{1}{b+1}\left(\ln(k-1)-\ln
    W+\ln(b+1)-\ln(1-p)-3\ln(1+\vareps)\right).\]
Letting
\[E_n:=\bigcap_{k=x_n}^\infty\left( E^1_k\cap E^2_k\right),\]
we have by the properties of $E^1_k$ and $E^2_k$ that
$\lim_{n\rightarrow\infty}\P\left(E_n\right)=1$.

On the event $E_n$, it holds by construction that for all $n$ large and
$i$ with $x_n\leq i\leq n$,
\[\tau_n-b_i\leq \frac{1}{b+1}\left(\ln n-\ln
    (i-1)+\ln(1-p)+3\ln(1+\vareps)-\ln(1-\vareps)\right).\]
Entirely similar, one sees that on $E_n$
\[\tau_n-b_i\geq \frac{1}{b+1}\left(\ln n-\ln
    (i+1)+\ln(1-p)+2\ln(1-\vareps)-2\ln(1+\vareps)\right).\]
Now notice that 
\[\max\left\{3\ln(1+\vareps)-\ln(1-\vareps),2\ln(1+\vareps)-2\ln(1-\vareps)\right\}\downarrow 0\]
if $\vareps\downarrow 0$. Since $\vareps>0$ can be chosen arbitrarily
small, we can clearly construct a sequence $(\vareps_n)$ with
$\vareps_n\downarrow 0$ such that on $E_n$, the stated bounds hold.
\end{proof}
\end{appendix}

\noindent{\bf Acknowledgments.} I warmly thank Silvia Businger for
explaining her work and for her help, and Jean Bertoin for
valuable comments. I am also grateful to two anonymous referees for their
careful reading of the manuscript and for their helpful remarks.

\end{document}